\documentclass[a4paper,11pt]{amsart}

\usepackage{amsfonts,amsmath,amssymb,amsthm}
\usepackage{latexsym}
\usepackage[svgnames]{xcolor}
\usepackage{epsfig}

\newtheorem{theorem}{Theorem}[section]
\newtheorem{claim}{Claim}[section]

\newtheorem{lemma}[theorem]{Lemma}
\newtheorem{proposition}{Proposition}

\theoremstyle{definition}
\newtheorem{definition}[theorem]{Definition}
\newtheorem{remark}{Remark}

\newcommand{\nwc}{\newcommand}
\nwc\eps{\varepsilon}

\nwc{\be}{\begin{equation}}
\nwc{\ee}{\end{equation}}
\nwc{\bee}{\begin{eqnarray}}
\nwc{\eee}{\end{eqnarray}}
\nwc{\ba}{\begin{array}}
\nwc{\ea}{\end{array}}

\nwc\bo{\mbox{Bo}}  
\nwc\calL{{\cal L}}
\nwc\calG{{\cal G}}

\newcommand\nd{\noindent}

\newcommand\I{{\mathcal{I}}}

\newcommand\R{\mathbb{R}}
\newcommand\T{\mathbb{T}}

\newcommand\C{\mathbb{C}}

\newcommand\Z{\mathbb{Z}}

\nwc\EN{{\mathcal{E}_1}}

\nwc{\dx}{\partial_x}
\nwc{\dy}{\partial_y}

\nwc{\hamone}{{\mathcal{H}}}
\nwc{\Imu}{\Sigma}
\nwc{\Iem}{\I}
\nwc{\Gep}{G}
\nwc{\wt}{\widetilde}

\newcommand{\sgn}{\mathop{\rm sgn}\nolimits}

\numberwithin{equation}{section}

\title[Controllability of a  long wave system]{The Well-posedness and Controllability of the Generalized Symmetric Regularized Long Wave System}

\author[Gallego]{F.A. Gallego}
\address{Departamento de Matem\'aticas y Estad\'istica, Universidad Nacional de Colombia - Sede Manizales, Colombia}
\email{fagallegor@unal.edu.co}
\author[Montes]{A.M. Montes}
\address{Universidad del Cauca}
\thanks{Research  supported by Universidad del Cauca and Colciencias}
 \email{amontes@unicauca.edu.co}


\begin{document}

\begin{abstract}
The symmetric regularized long wave system (SRLW) is a model for the weakly nonlinear ion acoustic and space-charge
waves, which was introduced by C. Seyler and D. Fenstermacher. 
In this paper, we investigated the wellposedness and controllability properties of the generalized  symmetric regularized long wave system (g-SRLW) in different structures (periodic and bounded domains).  Firstly, the wellposedness and the exact controllability results for both linear and nonlinear g-SRLW system posed on the one-dimensional torus are obtained under the effect of a distributed moving control. Second, we consider the g-SRLW system in a bounded interval with some Dirichlet-Neumann conditions and we show that the system is not spectrally controllable (No finite linear combination of eigenfunctions associated with the state equations, other than zero, can be steered to zero). Although the system is not spectrally controllable, it can be shown that it is approximately
controllable.
\end{abstract}

\maketitle

\section{ Introduction}

\subsection{Setting of the Problem}

The Symmetric Regularized Long Wave (SRLW) model  was first derived by C. Seyler and D. Fenstermacher in \cite{S-F} to describe the propagation  of the weakly nonlinear  ion acoustic and space-charge waves, shallow water waves and solitary waves with bidirectional propagation, where a weakly nonlinear analysis of the cold-electron fluid equation is made:
\begin{equation}
\begin{cases}
\partial_{t}u-\partial_t\partial_x^2u+u\partial_xu-\partial_xv=0, 
\\ \label{srlw} 
\partial_{t}v-\partial_xu=0, 
\end{cases}
\end{equation}
where $u=u(x,t)$ and $v=v(x,t)$ are the fluid velocity and the density, respectively. 
In this work we are concerned with the generalized Symmetric Regularized Long Wave (g-SRLW) system, more
precisely,
\begin{equation}
\begin{cases}
\partial_{t}u-\partial_t\partial_x^2u+u^p\partial_xu-\partial_xv=0, 
\\ \label{gsrlw1} 
\partial_{t}v-\partial_xu=0,  \\
u(x,0)=u_0(x), \quad v(x,0)=v_0(x).
\end{cases}
\end{equation}
The density function can be removed from \eqref{gsrlw1} and then the equations turn to a single nonlinear equation for the velocity function. Thus, eliminating $v$ from  \eqref{gsrlw1}, we get a class of g-SRLW equation:
\begin{equation}\label{eq1}
\partial_t^2 u-\partial_x^2 u-\partial_x^2\partial_t^2 u+\frac{1}{p+1}\partial_{t}\partial_x ( u^{p+1})=0.
\end{equation}
Equation \eqref{eq1} is explicitly symmetric in the $x$ and $t$ derivatives and is very similar to the
regularized long wave equation that describes shallow water waves and plasma drift waves \cite{albert, S-F}. The SRLW equation \eqref{gsrlw1} or \eqref{eq1} arises also in many other areas of mathematical physics.

Our main aim is to address two mathematical issues connected with the initial-value
problem \eqref{gsrlw1} posed in different structures: The well-posedness and controllability properties. The main difficulty in this context comes from the structure of nonlinearities and the lack of regularity of the solutions
we are dealing with. In the literature there are some papers about the SRLW equation from the mathematical point of view. It is a special case of a broad class of nonlinear evolution equations for which the well-posedness theory associated with the pure initial value problem posed on the whole real line $\R$, or on a finite interval with periodic boundary conditions, has been intensively
investigated. In the work \cite{Chen1}, by using semigroup theory, the author showed that the initial value problem associated to the system \eqref{srlw}  is global well-posed in the Sobolev type space $H^1(\R) \times L^2(\R)$. Later on, in \cite{CB} improved this result, establishing the global well-posedness in the space  $H^1(\R) \times H^{s-1}(\R)$ for  $s \geq 0$.  On the periodic case, in \cite{CB}, using a estimative obtained by D. Roumégoux in \cite{R}, the proved that system \eqref{srlw} with initial data $(u_0,v_0) \in H^s_{per}([0,T]) \times H^{s-1}_{per}([0,T])$   is globally well-posed for $s\geq 0$.

By contrast, the mathematical theory pertaining to the study of the controllability properties 
is considerably less advanced to the equation \eqref{srlw}.  The program of this work was carried out for a particular choice of controls functions and appropriated domains to establish interesting controllability properties. As far as we know there are no results for the controllability of the g-SRLW system \eqref{gsrlw1}.  In this direction, we can refer the work \cite{cerpa2018}, where the authors analyze the controllability properties of the improved Boussinesq equation posed on a bounded interval and one-dimensional torus, respectively. From the mathematical point of view, the improve Boussinesq and the SWRL equation \eqref{eq1} are similar, since both equations have the same linear system around zero. However, in general, it is not clear that the solution of linear equation  associated to \eqref{eq1} is also solution of the linear system associated to \eqref{gsrlw1} at least in some structures, for instance on bounded domains with certain boundary conditions. Then, we are concern directly with the well-posedness and controllability properties of the system \eqref{gsrlw1} and not in equation \eqref{eq1}. Thus, this work is devoted to prove a series of results concerning the controllability  properties of the g-SRLW equation \eqref{gsrlw1}  in different environments.  The purpose is to see whether one can force the solutions of those systems to have certain desired properties by choosing appropriate control inputs. Consideration will
be given to the following fundamental problems that arises in control and wellposedness theories:

\vglue 0.4cm

\emph{\textbf{Problem $\mathcal{A}$:}} 
{\em 
Initially, we consider the g-SRLW equation posed in a periodic domain, under effect of internal control posed on the torus.  More precisely, we want to answer the following question: Can we find a kind control $f = f(x, t)$
such that the solution of system 
\begin{equation}\label{p1}
\begin{cases}
\partial_{t}u-\partial_x^2\partial_tu+u^p\partial_x u-\partial_xv=f,  & \ \ \ x\in \T,\ \  t\in \R, \\
\partial_{t}v-\partial_xu=0, & \ \ \ x\in \T,\ \  t\in \R, \\
u(x,0)=u_0(x), \,\, v(x,0)=v_0(x), & \ \ \ x\in \T
\end{cases}
\end{equation}
is exactly  controllable?, it means one can always find control inputs to guide the system from any given initial state $(u_0, v_0)$ to any given terminal state $(u_T, v_T)$ in time $T>0$, i.e, given initial and final data $(u_0,v_0)$ and $(u_T,v_T)$, respectively, there exists a control $f$ in some appropriated space such that the corresponding solution $(u,v)$ of \eqref{p1}  satisfies 
\begin{equation*}
u(x,T)=u_{T}(x) \quad\text{and}\quad v(x,T)=v_{T}(x).
\end{equation*}}

Moreover, to complete the analysis, we study the case of the initial boundary value problem related to the controllability in a bounded interval with some boundary condition: 
\vglue 0.2cm 
\emph{\textbf{Problem $\mathcal{B}$:}} 
{\em Second, we are interesting to study  controllability aspects for the g-SRLW equation \eqref{gsrlw1} posed in a bounded interval. Is there some kind of boundary controllability for the g-SRLW in a bounded domain?, More precisely, for any $(u_0,v_0)$ and $(u_T,v_T)$, is there exist a control $h$ in some appropiated spaces such that the solution of:
\begin{equation}\label{p12}
\begin{cases}
\partial_{t}u-\partial_t\partial_x^2u+u^p\partial_x u-\partial_xv=0, & (x,t) \in (0,1)\times (0,T), \\ 
\partial_{t}v-\partial_xu=0,  & (x,t) \in (0,1)\times (0,T), \\
u(0,t)=0,  \quad u(1,t)=h(t), & t \in (0,T), \\
\partial_x v(0,t)=0,  \quad \partial_xv(1,t)=0, & t \in (0,T), \\
u(x,0)=u_0(x), 	\quad v(x,0)=v_0(x), & x \in (0,1),
\end{cases}
\end{equation}
satisfies
\begin{equation*}
u(x,T)=u_{T}(x) \quad\text{and}\quad  v(x,T)=v_{T}(x)?
\end{equation*}}
Before presenting an answer for problems $\mathcal{A}$  and $\mathcal{B}$, it is necessary to investigate the wellposedness of the full system \eqref{p1} and/or \eqref{p12}. Thus, the following issue appears naturally:
\vglue 0.3cm

\emph{\textbf{Problem $\mathcal{C}$:}} 
{\em 
Are the Nonhomogeneous Generalized Symmetric Long Wave Equation well-posed in time in the structures of problems $\mathcal{A}$ and/or $\mathcal{B}$, respectively?}
\vglue 0.2cm

Whit the above information in hand  and following the ideas contained in \cite{cerpa2018, micu2009, rosier}, we get some  answers to problems $\mathcal{A}$, $\mathcal{B}$ and $\mathcal{C}$, respectively.

\subsection{Main Results and Notations} 

We are now in position to present our main results. Hence, we will consider two cases corresponding to above problems, the system posed in the torus and a bounded interval with some boundary condition, respectively. 

\subsubsection{\textbf{Periodic Case}}
In order to give an answers to problems $\mathcal{A}$ and $\mathcal{C}$, respectively, we establish a moving local controllability for the system \eqref{p1}. Thus, let us introduce some Hilbert spaces. We denote by $X^s= X^s(T)$ the space defined on the torus $\T=\R /(2n\Z)$, such that for any $s \in \R$
and $T > 0$ fixed, we use the notation $X^s := H^s(\T) \times H^{s-2}(\T)$, where the periodic Sobolev space $H^s(\T)$ is given by
$$
H^s(\T)=\Bigl\{u(x)=\sum_{k\in \mathbb Z}u_ke^{ikx}  \  \    :  \   \  \|u\|_{H^s(\T)}^2= \sum_{k\in \mathbb Z}(1+|k|^2)^{s}  |u_k|^2 \,  < \, +\infty  \Bigr\}, 
$$
where $u_k=\widehat{u}(k)$ denotes the $k$-Fourier coefficient with respect to the spatial variable $x$. Thus, $X^s$ is a Hilbert space equipped with the norm 
\begin{equation}
\|(u,v)\|_{X^s}:= \left(\|u\|_{H^s(\T)}^2+ \|v\|_{H^{s-2}(\T)}^2\right)^{1/2}.
\end{equation}

In order to give answer to problem $\mathcal{A}$,  we shall consider several control problems in a periodic domain related to the moving controllability. The bad control properties from g-SRLW system come from the existence of a finite accumulation point in the spectrum of the differential operator associated to equation. Note that the bad control properties come from the term $-\partial_x^2\partial_tu$, hence it is natural to ask
whether better control properties for g-SRLW system could be obtained by using a moving control. Such a phenomenon was noticed first by D. Russell in \cite{rusell1985} for the beam equation
with internal damping, by G. Leugering and E. J. P. G. Schmidt in \cite{leu1989} for the plate equation with internal damping, and by S. Micu in \cite{micu} for the linearized Benjamin-Bona-Mahony (BBM) equation. The concept of moving point control was introduced by J. L. Lions in \cite{lions_3}  for the wave equation. One important motivation for this kind of control is that the exact controllability of
the wave equation with a pointwise control and Dirichlet boundary conditions fails if the point is
a zero of some eigenfunction of the Dirichlet Laplacian, while it holds when the point is moving
under some (much more stable) conditions easy to check (see e.g. \cite{castro}).

Initially, we are concern with the exact controllability  with the moving  distributed control  of the nonlinear system posed on a periodic domain and it read as 
\begin{equation}\label{moving1}
\begin{cases}
\partial_{t}u-\partial_x^2\partial_tu+u^p\partial_x u-\partial_xv=b(x+ct)h(x,t),  & \ \ \ x\in \T,\ \  t\in \R, \\
\partial_{t}v-\partial_xu=0, & \ \ \ x\in \T,\ \  t\in \R, \\
u(x,0)=u_0(x), \,\, v(x,0)=v_0(x), & \ \ \ x\in \T.
\end{cases}
\end{equation}
where $c \in \R\setminus \{0\}$,  $b=b(x) \in C^{\infty}(\T)$ and $h$ is the control function. This control time has to be chosen in such a way that the support of the control, which is moving at the constant velocity $c$, can visit all the domain $\T$.  More precisely, in terms of controllability of system \eqref{moving1}, we have the following definition:
\begin{definition}
Let $T > 0$. System \eqref{lp} is exact controllable in time $T$ if for any initial and final data $(u_0,v_0)$ and $(u_T,v_T)$ in $X^s$, there exists control function  $h$ in $L^2(0,T;H^{s-2}(\T))$,  such that the solution satisfies  
\begin{equation}\label{control}
u(x,T)=u_T(x) \quad \text{and} \quad v(x,T)=v_T(x).
\end{equation}
\end{definition}
\vglue 0.2cm
The main result of the exact controllability of system \eqref{moving1} is the following:

\begin{theorem}\label{mainmovilcontrol}
 Suppose $s\geq 0$ if $p=1$ and  $s>\frac12$ if $p>1$. Let $b=b(x) \in C^{\infty}(\T)$ be such that 
$\omega =\left\lbrace x \in \T: b(x)\neq 0 \right\rbrace \neq \emptyset,$
 $|c|>2$ and $T > \frac{2\pi}{\Delta}$ where    
\begin{align*}
\Delta = \liminf_{k \rightarrow \infty} \left( \mu_{k+1}^{\pm}-\mu_k^{\pm}\right), \quad \mu_k^{\pm} = ck \pm \frac{|k|}{\sqrt{1+k^2}}.
\end{align*} 
Then, there exists a $\delta >0$ such that if 
\begin{equation*}
\|(u_0, v_0)\|_{X^s}+ \| (u_T, v_T)\|_{X^s} < \delta, 
\end{equation*}
one can find a control input  $ h\in L^2\left(0,T; H^{s-2}(\T)\right)$ such that  the system \eqref{moving1} admits a unique solution $(u,v) \in C\left([0,T], X^{s}\right)$ satisfying
\begin{equation}\label{contcond}
  u(x, T)= u_T(x) \quad \text{and} \quad  v(x,T)= v_T(x).
\end{equation}
\end{theorem}

This controllability result is based in a spectral analysis of the differential operator associated to the model in direction to resolve a moment problem for the linear SRLW system.  Next, we prove the exact controllability results for the full nonlinear system guarantying the existence of a bounded linear operator $\Phi$ from the initial/end state pair $(u_0,
v_0)$, $(u_{ T},  v_{T})$, each in the space $X^s$, to the
corresponding control in the space $L^2([0, T],H^{s-2}(\T))$,  by imposing smallness of the initial and terminal states. 

Moreover, we turn our attention to some internal control acting on a single moving point:

\begin{equation}\label{moving2}
\begin{cases}
\partial_{t}u-\partial_x^2\partial_tu-\partial_xv=g(t)\delta(x+{ct}),  & \ \ \ x\in \T,\ \  t\in \R, \\
\partial_{t}v-\partial_xu=0, & \ \ \ x\in \T,\ \  t\in \R, \\
u(x,0)=u_0(x), \,\, v(x,0)=v_0(x), & \ \ \ x\in \T.
\end{cases}
\end{equation}
where \(\delta\left(x-x_{0}\right)\) represents the Dirac measure at \(x_{0}\), defined by \(\left\langle\delta\left(x-x_{0}\right), \varphi\right\rangle=\varphi\left(x_{0}\right)\) for any test function \(\varphi .\) Here and in what follows, \(\langle\cdot, \cdot\rangle\) stands for the duality pairing \(\langle\cdot, \cdot\rangle_{\mathcal{D}^{\prime}(\mathbb{T}), \mathcal{D}(\mathbb{T})} \cdot\) In particular, \(\langle f, \varphi\rangle=\int_{\mathbb{T}} f(x) \varphi(x) d x\) for any \(f \in L^{1}(\mathbb{T})\) and any \(\varphi \in C^{\infty}(\mathbb{T}).\) Then, we will obtain the following result related to the exact controllability of the system \eqref{moving2}.

\begin{theorem}\label{mainmovilcontrol2}
For any time \(T>2 \pi\) and  any $(u_0,v_0), (u_T,v_T) \in H^{-1}(\T)\times L^2(\T)$, there exists a control \(g \in L^{2}(0, T)\) such that the solution of \eqref{moving2} satisfies 
$$u(T,x)=u_T(x) \quad \text{and} \quad v(T, x)=v_T(x).$$
\end{theorem}

Again, the exact controllability of the linearized system \eqref{moving2} are proved by using the moment method and spectral analysis. In order to get the same result for the nonlinear equation, we use a fixed point argument in Hilbert spaces. Finally, in order to give answer to problem $\mathcal{C}$, the well-posedness of the two problems is also established. Indeed,  we prove the local well-posedness of the nonlinear nonhomegeneus system in $X^s$, by using the Banach fixed  point Theorem and appropriate linear and nonlinear estimates with the bilinear estimative obtained by D. Roum\'egoux in \cite{R}.

\subsubsection{\textbf{Bounded Case}}

In this part, we address the problems described in the previous subsection and our main results provide a   answer for the problem $\mathcal{B}$.  In this sense, the approximate controllability and the lack of exact controllability are
proved. Hence,  we are concerned with the boundary controllability properties of the linearized symmetric regularized long wave system  posed on the finite interval $[0, 1]$. Given a time $T > 0$, we consider linear  problem:

\begin{equation}\label{linear1}
\begin{cases}
\partial_{t}u-\partial_t\partial_x^2u-\partial_xv=0, & (x,t) \in (0,1)\times (0,T), \\ 
\partial_{t}v-\partial_xu=0,  & (x,t) \in (0,1)\times (0,T), \\
u(x,0)=u_0(x), 	\quad v(x,0)=v_0(x), & x \in (0,1),
\end{cases}
\end{equation}
with Diriclhet-Neumann boundary conditions
\begin{equation}\label{linear2}
\begin{cases}
u(0,t)=0,  \quad u(1,t)=h(t), & t \in (0,T), \\
\partial_x v(0,t)=0,  \quad \partial_x v(1,t)=0, & t \in (0,T). \\
\end{cases}
\end{equation}
Hence, in order to establish some controllability aspects, we would like to found  a control function $h = h(t)$ on an adequate space. 

We begin our analysis by providing a negative result for the problem $\mathcal{B}$ introduced
above: The system \eqref{linear1}-\eqref{linear2} is not spectrally controllable if $(u_0,v_0) \in H^1_0(0,1)\times L^2(0,1)$.  This means that no finite linear nontrivial combinations of eigenvector of the differential operator associated to \eqref{linear1} can be driven to zero in finite time by using a control $h$. Thus, the lack controllability of the system \eqref{linear1}-\eqref{linear2} breaches the controllability principle. As it will become clear during our proofs, the bad control property comes from the existence of a limit point in the
spectrum of the operator associated with the state equations. To obtain the results, we make use of the
careful spectral analysis developed in next section, which provides important developments to
justify the use of eigenvector expansions for the solutions, as well as, the asymptotic
behavior of the eigenvalues.

Thus, we establish the following result related to the lack controlalbility of the linear SRLW system posed in a bounded interval:

\begin{theorem}\label{noncontrollable}
The control system \eqref{linear1}-\eqref{linear2} is not spectrally controllable in $ H^1_0(0,1)\times L^2(0,1).$
\end{theorem}

Nevertheless, in spite of the lack of exact controllability from the boundary, we will prove that system \eqref{linear1} is approximately controllable, i,e the set of reachable states 
\begin{align*}
R(T,(u_0,v_0))=\left\lbrace (u(T,x),v(T,x)): h \in H^1(0,T)\right\rbrace
\end{align*}
is dense in $L^2(0,1) \times L^2(0,1)$ for any $(u_0,v_0) \in H^{-1}(0,1) \times L^2(0,1)$ and $T>0.$  This property is equivalent to a unique continuation property for the adjoint system associated. In \cite{shang}, the author proved the UCP for the g-SRLW equation. As in the above problem, to obtain the results, we rely strongly on the carefully spectral analysis developed for the operator associated with the state equations. The main idea is to use the series expansion of the solution in terms of the eigenvectors of the operator in order to reduce the problem to a unique continuation problem (of the eigenvectors).

\begin{theorem}\label{approximated}
The System  \eqref{linear1} is approximately controllable in $H^1_0(0, 1) \times L^2(0,1)$ for any time $T > 0$.
\end{theorem}

Observe that exact controllability is essentially stronger notion than approximate controllability. In other words, exact controllability always implies approximate controllability. The converse statement is generally false

The paper is organized as follows. Then Section 2 is dedicated to study the well posedness and the moving exact controllability of the g-RSLW equation posed in a periodic domain, respectively, by using a Banach fixed  point Theorem, spectral analysis and the  moment method. Here, we impose smallness of the initial and terminal states. In section 3,  we consider the linearized SRLW system posed on a bounded domain with a boundary control. Firstly, by using the spectral properties of the differential operator associated to $\eqref{linear1}$ and the asymptotic behavior of its eigenvalues, we will able to guaranties the wellposedness in $ H^1(0, 1) \times  L^2(0, 1)$. Finally, the lack of exact controllability and the approximate controllability   are proved.

\section{Exact Controllability on a Periodic Domain.} \label{section2}

Before to present the above controllability results, we show the well-posedness of the systems \eqref{moving1} and \eqref{moving2}, respectively.

\subsection{\bf Well-posedness}
In this section we establish the well-posedness of the initial value
problem of the forcing generalized long wave model on the
periodic domain $\T$, 
\begin{equation}\label{ivp}
\left\{\begin{array}{rl}
\partial_{t}u-\partial_x^2\partial_tu+u^p\partial_xu-\partial_xv&=f,  \ \ \ x\in \T,\ \  t\in \R, \\
\\
\partial_{t}v-\partial_xu&=0,
\end{array}   \right.
\end{equation}
with the initial condition
\begin{equation}\label{ic}
u(x,0)=u_0(x),  \  \  v(x,0)=v_0(x),
\end{equation}
via the contraction principle approach.  Indeed, we see that the system (\ref{ivp}) can be rewritten as

\begin{equation}\label{fors}
\partial_tU=AU-G(U)+ F, \  \   \    U=(u,v)^t,
\end{equation}
where
\[
A=\begin{pmatrix}0 & \left(I-\partial_x^2\right)^{-1}\partial_{x}\\ \\
\partial_{x} & 0
       \end{pmatrix}, \ \
 G(U) =\begin{pmatrix} \left(I-\partial_x^2\right)^{-1}\partial_x\bigl(\frac1{p+1} u^{p+1}\bigr)\\ \\
                    0
 \end{pmatrix}, \  \   F=\begin{pmatrix} (I-\partial_x^2)^{-1}f\\ \\ 0 \end{pmatrix}.
\]

Now, if we consider the Sobolev type space $X^s=H^{s}(\T)\times
H^{s-1}(\T) $ with norm given by
$$
\|(u,v)\|^2_{X^s
}=\|u\|^2_{H^{s}}+\|v\|^2_{H^{s-1}},
$$
then  $A:X^s \rightarrow X^{s} $ is a
bounded linear operator. In fact, for
 $U=(u,v)\in X^s $ we see that
\begin{align*}
\|AU\|^2_{X^s}&=\|\partial_xu\|^2_{H^{s-1}}+\|(I-\partial_x^2)^{-1}\partial_xv\|^2_{H^{s}} \\
&\leq \|u\|^2_{H^{s}}+\|v\|^2_{H^{s-1}}\\
&=\|U\|^2_{X^{s}}.
\end{align*}
We also have the following estimates on $G$.
\begin{lemma}\label{pG}
 Suppose $s\geq0$ if $p=1$ and $s>\frac12$ if
$p>1$, then there are constants $C_2, C_3>0$ such that
\begin{enumerate}
\item $\|G(U)\|_{X^s}\leq C_2\|U\|^{p+1}_{X^s}.
$

\medskip

\item $\|G(U)-G(V)\|_{X^s}\leq C_3\left(\|U\|_{X^s}+\|V\|_{X^s}\right)^p\|U-V\|_{X^s}$.
\end{enumerate}
\end{lemma}

\begin{proof}
The proof is achieved using the following result obtained by D.
Roumegoux  in \cite{R}: if $s\geq0$ then there exists a constant
$C_s>0$ such that for all $u,v\in H^s(\T)$,
$$
\|(I-\partial_x)^{-1}\partial_x(uv)\|_{H^{s}(\T)}\leq
C_s\|u\|_{H^{s}(\T)}\|v\|_{H^{s}(\T)}.
$$
\end{proof}

Now, in order to consider the initial value problem we need to describe the semigroup $S(t)$ associated with the linear problem
\begin{equation*}\label{lis}
U_t+AU=0.
\end{equation*}
A simple calculation shows that the unique solution of the previous linear problem  with the
initial condition
\begin{equation}\label{icee}
U(0, \cdot)=(u(0,\cdot), v(0,\cdot))=(u_0,v_0)=U_0 \in X^s,
\end{equation}
is given by
\begin{equation*}
U(t)=(u(t), v(t))=S(t)U_0,
\end{equation*}
where $S(t)$ is defined as
\begin{equation}\label{semig}
\widehat{S(t)(U)}(k)=\begin{pmatrix}
\cos\left(\rho(k)t\right) & \frac{i}{\sqrt{1+k^2}}\sin(\rho(k)t)\\ \\
i\sqrt{1+k^2}\sin\left(\rho(k)t\right) & \cos\left(\rho(k)t\right)
\end{pmatrix}\widehat{U}(k),
\end{equation}
and the function $\rho$ is given by
\begin{equation*}
\rho(k)=\sqrt{\frac{k^2}{1+k^2}}.
\end{equation*}
It is convenient to set
\begin{equation*}
Q(t)(\widehat u, \widehat v)=\bigl(Q_1(t), Q_2(t)\bigr)(\widehat u,
\widehat v),
\end{equation*}
where
\begin{align*}\label{Q1Q2}
\left[Q_1(t)(\widehat u, \widehat
v)\right](k)&=\cos\left(\rho(k)t\right)\widehat u +\frac{i}{\sqrt{1+k^2}}\sin(\rho(k)t) \widehat v\\
\left[Q_2(t)(\widehat u, \widehat v)\right](k)& = i\sqrt{1+k^2}\sin\left(\rho(k)t\right)\widehat u+\cos\left(\rho(k)t\right)\widehat v.\nonumber
\end{align*}
Then we have that
\begin{equation}\label{solutionlinear}
S(t)(U)=\left(\sum_{k \in\mathbb Z}\bigl[Q_1(t)(\widehat
U)\bigr]e^{ikx},
 \sum_{k \in\mathbb Z}\bigl[Q_2(t)(\widehat U)\bigr]e^{ikx}\right),
\end{equation}
and it is easy to prove the following result on $S(t)$.

\begin{lemma}\label{lebbl}
Suppose $s\in\R$. Then for all $t\in\R$, $S(t)$ is a bounded linear
operator from $X^s$ into $X^s$. Moreover, there exists $C_1>0$ such
that for all $t\in\R$,
\begin{equation*}
\|S(t)(U)\|_{X^s}\leq C_1\|U\|_{ X^s}.
\end{equation*}
\end{lemma}

Next, using the Banach Fixed Point Theorem, we establish the well-posedness of the initial value problem of the forcing regularized long wave model.

\begin{theorem}\label{locexis}
Let $T^*>0$ be given. Suppose $s\geq0$ if $p=1$ and $s>\frac12$ if
$p>1$, then for all $(u_0, v_0)\in X^s$ and $f\in L^1\left(0,T^*;
H^{s-2}(\T)\right)$ there exists a time $T>0$ which depends only on
$\|(u_0,v_0)\|_{X^s}$ and $\|f\|_{L^1\left(0,T^*;
\,H^{s-2}(\T)\right)}$ such that the problem (\ref{ivp})-(\ref{ic}) has
a unique solution $(u,v)\in C\left([0,T], X^s \right).$ Moreover,
for all $0<T_1<T$ there exists a neighborhood $\,\mathbb{V}$ of
$\,(u_0, v_0, f)$ in $X^s \times L^1\left(0,T^*; \,H^{s-2}(\T)\right)$
such that the correspondence $(\tilde u_0,\tilde v_0, \tilde
f)\longrightarrow(\tilde u(\cdot),\tilde v(\cdot))$, that associates
to $(\tilde u_0,\tilde v_0, \tilde f)$ the solution $(\tilde
u(\cdot), \tilde v(\cdot))$ of the  Cauchy problem (\ref{ivp}) with
initial condition $(\tilde u_0,\tilde v_0)$ and control $\tilde f$,
is a Lipschitz mapping from $\mathbb{V}$ in $C([0,T_1], X^s)$.
\end{theorem}
\begin{proof}
Remember that the initial value problem (\ref{ivp})-(\ref{ic}) can
be written as the equivalent first order evolution system
(\ref{fors})-(\ref{icee}).  It is known that the Duhamel's principle
implies that if $U$ is a solution of (\ref{fors}) with the initial
condition (\ref{icee}), then this solution satisfies the integral
equation
\begin{equation}\label{ie}
U(t)=S(t)U_0-\int_0^tS(t-\tau)G(U)(\tau)d\tau+\int_0^tS(t-\tau)F(\tau)d\tau.
\end{equation}
Now, given $T>0$  we define the Banach space
$Y^s(T)=C([0,T],X^s)$, equipped with the norm defined by
\begin{equation}
\|U\|_{Y^s(T)}=\max_{t\in[0,T]}\|U(\cdot,t)\|_{X^s}.
\end{equation}
Let $B_R(T)$ be
the closed ball of radius $R$ centered at the origin in $Y^s(T)$,
i.e.
$$
B_R(T)=\left\{U\in Y^s(T)\,\,:\,\,\|U\|_{Y^s(T)}\leq R\right\}.
$$
For fixed $U_0=(u_0, v_0)\in X^s$ and $f\in L^1\left([0,T^*], \,H^{s-2}(S)\right)$,  we define the map
\begin{equation}
\Psi(U(t))=S(t)U_0-\int_0^tS(t-\tau)G(U(\tau))d\tau+\int_0^tS(t-\tau)F(\tau)d\tau,
\end{equation}
where $U=(u, v)\in Y^s(T)$. We will show that the correspondence
$\,U(t)\mapsto\Psi(U(t))\,$ maps $B_R(T)$ into itself and is a
contraction if $R$ and $T$ are well chosen. In fact, if $t\in[0,T]$
and $U\in B_R(T)$, then using Lemma \ref{lebbl} and statement
(\textit{1}) of Lemma \ref{pG} we have that
\begin{align*}
\|\Psi(U(t))\|_{X^s}&\leq
C_1\left(\|U_0\|_{X^s}+\|f\|_{L^1\left(0,T^*; \,H^{s-2}\right)}+C_2\int_0^t\|U(\tau)\|_{X^s}^{p+1}\,d\tau\right)\\
&\leq C_1\left(\|U_0\|_{X^s}+\|f\|_{L^1\left(0,T^*; \,H^{s-2}\right)}+C_2R^{p+1}T\right).
\end{align*}
Choosing
$R=2C_1\left(\|U_0\|_{X^s}+\|f\|_{L^1\left(0,T^*;H^{s-2}\right)}\right)$
and $0<T\leq T^*$ such that
\begin{equation}\label{t1}
2C_1C_2R^{p}\,T\leq1,
\end{equation}
we obtain that
\begin{align*}
\|\Psi(U(t))\|_{X^s}&\leq
C_1\,\left(\|U_0\|_{X^s}+\|f\|_{L^1\left(0,T^*; \,H^{s-2}\right)}\right)(1+2C_1C_2R^p\,T)=
R.
\end{align*}
So that $\Psi$ maps $B_R(T)$ to itself. Let us prove that $\Psi$ is
a contraction. If $U,V\in B_R(T)$, then by the definition of $\Psi$
we have that
\begin{equation*}
\Psi(U(t))-\Psi(V(t))=-\int_0^tS(t-\tau)\bigl[G(U(\tau))-G(V(\tau))\bigr]d\tau.
\end{equation*}
Then using the statement (\textit{2}) of Lemma \ref{pG} we see
that for $t\in[0,T]$,
\begin{align*}
\left\|\Psi(U(t))-\Psi(V(t))\right\|_{X^s}&\leq
C_1C_3\int_0^t\left(\|U(\tau)\|_{X^s}
+\|V(\tau)\|_{X^s}\right)^p\|U(\tau)-V(\tau)\|_{X^s}\,d\tau\\
&\leq C_1C_3(2R)^p\,T\|U-V\|_{Y^s(T)}.
\end{align*}
We choose $T$  enough small so that (\ref{t1}) holds and
\begin{equation}\label{t2}
C_1C_3(2R)^p\,T \leq\frac12.
\end{equation}

So, we conclude that
\begin{align*}
\left\|\Psi(U)-\Psi(V)\right\|_{Y^s(T)}&\leq \frac12\|U-V\|_{Y^s(T)}.
\end{align*}
Therefore $\Psi$ is a contraction. Thus, there exists a unique fixed
point of $\Psi$ in $B_R(T)$, which is a solution of the integral
equation (\ref{ie}).  Now, if $(u(t),v(t))\in C([0,T], X^s)$ is a
solution of the integral equation (\ref{ie}), obviously
$(u(0), v(0))=(u_0, v_0)$. Moreover, differentiating the
equation (\ref{ie}) with respect to $t$, there appears the relation
(\ref{fors}). Then $(u(t), v(t))$ is a solution of the
considered  problem.

The uniqueness of the solution and the local Lipschitz continuity
are consequence of the following standard argument. Let
$U(t)=(u(t), v(t)), \, V(t)=(\tilde u(t),\tilde v(t))$ be
two solutions of the equation (\ref{fors}) with condition
$U_0=(u_0,v_0, f)$ and $V_0=(\tilde u_0,\tilde v_0, \tilde f)$
respectively, satisfying
$$
U, V\in C([0,T], X^s).
$$
Then, if $t\in[0,T]$ we see that
\begin{align*}
\|U(t)-V(t)\|_{X^s}&\leq C_1\left(\|(U_0-V_0)\|_{X^s}+\|f-\tilde f\|_{L^1\left(0,T^*; \,H^{s-2}\right)}\right)\\
&+ C_1C_3\int_0^t\left(\|U(\tau)\|_{X^s}
+\|V(\tau)\|_{X^s}\right)^p\|U(\tau)-V(\tau)\|_{X^s}\,d\tau\\
&\leq
C_1\left(\|U_0-V_0\|_{X^s}+\|f-\tilde f\|_{L^1\left(0,T^*; \,H^{s-2}\right)}\right)\\
&+(2N)^pC_1C_3\int_0^t\|U(\tau)-V(\tau)\|_{X^s}\,d\tau,
\end{align*}
where we assume that
$$
\|U\|_{X^s(T)}, \|V\|_{X^s(T)}\leq N.
$$
Then, taking $C_4=2^p\,C_1C_3$ and using Gronwall's inequality, we have
for all $t\in[0,T]$ that
\begin{equation}\label{lc}
\|U(t)-V(t)\|_{X^s}\leq C_1e^{C_4N^p\,T}\left(\|U_0-V_0\|_{X^s}+\|f-\tilde f\|_{L^1\left(0,T^*; \,H^{s-2}\right)}\right).
\end{equation}
Now, if $U_0=V_0$ and $f=\tilde f$ then we conclude that $\|U-V\|_{Y^s(T)}=0$, which
yields the uniqueness.

On the other hand,  from the discussion of the existence of
solutions for the problem (\ref{fors})-(\ref{icee}),  for
$(U_0,f)\in X^s\times L^1\left(0,T^*; \,H^{s-2}(\T)\right)$ we have  that the existence time $T$ of
the associated solution is characterized by the inequalities
(\ref{t1})-(\ref{t2}). Then, if $\,0<T_1<T\,$ we define
$$
\mathbb{V}=\left\{ (V_0, \tilde f) \in X^s\times  L^1\left(0,T^*; \,H^{s-2}(\T)\right)    \     :     \
\|V_0-U_0\|_{X^s}+\|\tilde f- f\|_{L^1\left(0,T^*; \,H^{s-2}\right)}<\epsilon\right\}
$$
with $\epsilon=\left(\frac{1}{T_1^{1/p}}-\frac1{T^{1/p}}\right)\epsilon'$ and
$0<\epsilon'<\min\left\{\frac{1}{2C_1\left(2C_1C_2\right)^{1/p}},\,\frac{1}{4C_1\left(2C_1C_3\right)^{1/p}}\right\}.$ Note that if $(V_0, \tilde f)\in\mathbb{V}$ then we see that
\begin{equation}\label{v0}
\|V_0\|_{X^s}+\|\tilde f\|_{L^1\left(0,T^*; \,H^{s-2}\right)}\leq\epsilon+\|U_0\|_{X^s}+\| f\|_{L^1\left(0,T^*; \,H^{s-2}\right)}.
\end{equation}
Hence we have that
\begin{equation}\label{tprime}
2C_1C_2(R')^p\,T_1\leq1,\,\,\,\,2C_1C_3(2R')^p\,T_1\leq1,
\end{equation}
where $R'=2C_1\left(\|V_0\|_{X^s}+\|\tilde f\|_{L^1\left(0,T^*;
\,H^{s-2}\right)}\right)$. In fact,
\begin{align*}
2C_1C_2(R')^p&\leq 2C_1C_2\left( 2C_1\epsilon+R\right)^p\\
&\leq\biggl(2C_1\left(2C_1C_2\right)^{1/p}\biggl(\frac{1}{T_1^{1/p}}-\frac1{T^{1/p}}\biggr)\epsilon'+\left(2C_1C_2\right)^{1/p}R\biggr)^p<\frac{1}{T_1}.
\end{align*}
In a similar way we have the other inequality. Then from
(\ref{tprime}) we conclude that for any
$(V_0, \tilde f)=(\tilde u_0,\tilde v_0, \tilde f)\in\mathbb V$ there exists a unique
solution $V(t)=(\tilde u(t),\tilde v(t))\in C([0,T_1],X^s)$ of
the problem (\ref{fors}) with initial condition $V_0$ and control $\tilde f$; moreover,
$V\in B_{R'}(T_1)$ with $R'=2C_1\left(\|V_0\|_{X^s}+\|\tilde f\|_{L^1\left([0,T^*],
\,H^{s-2}\right)}\right)$. Then using (\ref{v0}),
for $(V_0, \tilde f), (W_0, \tilde g)\in\mathbb V$ with associated solutions $V$ and $W$
respectively, we have that
\begin{align*}
\|V\|_{Y^s(T)}+\|W\|_{Y^s(T)}&\leq2C_1\left(\|V_0\|_{X^s}+\|\tilde f\|_{L^1\left(0,T^*;
\,H^{s-2}\right)}+\|W_0\|_{X^s}+\|\tilde g\|_{L^1\left(0,T^*;
\,H^{s-2}\right)}\right)\\
&\leq 4C_1\left(\epsilon+\|U_0\|_{X^s}+\| f\|_{L^1\left(0,T^*;\,H^{s-2}\right)}\right)\\
&\leq\frac{2}{\left(2C_1C_2\right)^{1/p}}\biggl(\frac{1}{T_1^{1/p}}-\frac1{T^{1/p}}\biggr)+4C_1\left(\|U_0\|_{X^s}+\| f\|_{L^1\left(0,T^*;
\,H^{s-2}\right)}\right)\\
&\leq\frac{2}{\left(2C_1C_2T_1\right)^{1/p}}.
\end{align*}
Then using $(\ref{lc})$ with $N=2\left(2C_1C_2T_1\right)^{-1/p}$ we obtain that for
all $t\in[0,T_1]$,
\begin{align*}
\|V(t)-W(t)\|_{Y^s}&\leq C_1e^{C_4NT_1}\|V_0-W_0\|_{Y^s} \leq
C\|V_0-W_0\|_{Y^s}
\end{align*}
where $C=C_1e^{2^{2p}C_3(2C_2)^{-1}},$ and the proof is complete.
\end{proof}

Now, with the Theorem \ref{locexis} in hands, we turn to establish the wellposednes of the linear system \eqref{moving2} via transposition, i.e, the equation \eqref{moving2}  has to be understood in a weak sense. For instance, the respective solution can be defined by transpositions (see \cite{lions_1}, \cite{lions_2}). Let us briefly recall how can this be done.

Consider $f \in L^1(0, T, H^{-1}(\T))$,  $g \in L^1(0, T, L^2(\T))$and $(\varphi,\psi)$ the solution of the adjoint equation
\begin{equation}\label{adjointnew}
\begin{cases}
\partial_{t} \varphi-\partial_x^2\partial_t \varphi-\partial_x \psi=f,  & \ \ \ x\in \T,\ \  t\in \R, \\
\partial_{t}\psi-\partial_x \varphi=g, & \ \ \ x\in \T,\ \  t\in \R, \\
\varphi(x,T)=0, \,\, \psi(x,T)=0, & \ \ \ x\in \T.
\end{cases}
\end{equation}
The wellposedness of system \eqref{adjointnew} is guaranties by using a similar argument to the Theorem \ref{locexis}. Thus, there exists a unique solution $(\varphi,\psi)$ belongs to $C([0,T];X^{1})$. To establish the solution via transposition, we  multiply  \eqref{moving2}  by $(\varphi,\psi)$ and integrating by parts (formally)
\begin{multline}
\int_0^T\int_{\T} u f dxdt + \int_0^T\int_{\T} v g dxdt     -  \int_{\T} u_0(x)\varphi(x,0)   -\int_{\T} \partial_x u_0(x) \partial_x \varphi(x,0)dx  \\ -  \int_{\T}  v_0(x) \psi(x,0)dx =   \int_0^T \int_{\T} g(t)\varphi(x,t) \delta(x+ct)dxdt.
\end{multline}
Therefore, we can say that $(u,v)$ is the solution of \eqref{moving2}   if and only if
\begin{multline}\label{transposition}
\int_0^T ( u(t), f(t))_{L^2(\T)}dt + \int_0^T   (v(t), g(t) )_{L^2(\T)}dt     \\
-  \left\langle u_0, \varphi(0) \right\rangle_{H^{-1}(\T), H^1(\T)}     -  (  v_0, \psi(0) )_{L^2(\T)}  =   \int_0^T  g(t)\varphi(-ct,t)dt
\end{multline}
for all $f \in L^1(0, T;H^{-1}(\T))$, $g \in L^1(0, T;L^2(\T))$ and $(\varphi, \psi)$ the solution of \eqref{adjointnew}.  As in \cite{lions_1}, \cite{lions_2} it can be proved that for any initial data $(u_0,v_0) \in H^{-1}(\T) \times L^2(\T)$, the identity \eqref{transposition} has a unique solution
$(u.v) \in C([0, T];H^{-1}(\T) \times L^2(\T))$.


\subsection{\bf Spectral Analysis}

We analyze the linear operator related to the g-SRLW system and study its
spectral structure, eigenvalues and eigenfunctions, which will play an important role
in the characterization of the solutions. Then, in this section we perform the spectral analysis for the operator
\[
A=\begin{pmatrix}0 & (I-\partial_x^2)^{-1}\partial_{x}\\ \\
\partial_{x} & 0
\end{pmatrix},
\]
defined in the space $X^s$. The result in this analysis will be the 
basis to transfer the exact controllability of the associated linear
system  to the nonlinear system. Let us define
\[
E_{1, k}= \begin{pmatrix}e^{ikx} \\  0
\end{pmatrix}, \ \ \ E_{2, k}= \begin{pmatrix}0 \\ ke^{ikx}
\end{pmatrix},
\]
for $ k\in \Z^*=\Z\setminus\{0\}$. If we set
\[
\Sigma_k=\begin{pmatrix}0 & \frac{i k^2}{1+k^2} \\
i   & 0
\end{pmatrix},  \  \    k\in \Z^*,
\]
then we see directly that
\[
A(E_{1, k}, E_{2, k} )= (E_{1, k}, E_{2, k} )\Sigma_k,  \  \    k\in \Z^*.
\]

Moreover, we have that the eigenvalues for $\Sigma_k$ are
\[
\lambda_{ k}^+= i\, \sqrt{\frac{k^2}{1+k^2}}, \ \ \ \lambda_{k}^-=
-i\, \sqrt{\frac{k^2}{1+k^2}},  \  \    k\in \Z^*,
\]
with corresponding eigenvector
\[
 \tilde e_{k}^+= \begin{pmatrix}1 \\ \frac{(1+k^2)\lambda_k^+}{ik^2}
\end{pmatrix} , \ \ \ \tilde e_{ k}^-= \begin{pmatrix}1 \\   \frac{(1+k^2)\lambda_{k}^-}{ik^2}
\end{pmatrix}, \  \   k\in \Z^{*}.
\]
Thus, we have that
\begin{align*}
A(E_{1, k}, E_{2, k} )(\tilde e_{k}^+, \tilde e_{k}^- )& = (E_{1, k}, E_{2, k} )\Sigma_k(\tilde e_{k}^+, \tilde e_{k}^- )\\
&= (\lambda_{1, k}(E_{1, k}, E_{2, k} )\tilde e_{ k}^+, \lambda_{2,
k}(E_{1, k}, E_{2, k} )\tilde e_{ k}^-), \  \  k\in \Z^*,
\end{align*}
meaning that $\lambda_{k}^+$ and $\lambda_{k}^-$ are the eigenvalues for the operator $A$ with corresponding eigenvectors
\[
\eta_{k}^{\pm}= (E_{1, k}, E_{2, k} )\tilde e_{ k}^{\pm}, \ \, \ \
k\in \Z,
\]
where
\[
\lambda_{0}^+=\lambda_{0}^-=0,   \  \  \tilde e_{ 0}^+=\begin{pmatrix}1 \\
0
\end{pmatrix}, \  \   \tilde e_{ 0}^-= \begin{pmatrix}0 \\ 1
\end{pmatrix} , \ \ \
E_{1, 0}= \begin{pmatrix}1 \\  0
\end{pmatrix} , \ \ \ E_{2,0 }= \begin{pmatrix}0 \\  1
\end{pmatrix}.
\]
 On the other hand, we see that
\[
\lim_{k\to  \infty} (\tilde e_{k}^+, \tilde e_{k}^-)= \begin{pmatrix} 1 & 1\\  1 & -1\end{pmatrix}
\]
and also that
\[
\lim_{k\to \pm \infty} \det  (\tilde e_{k}^+, \tilde e_{k}^-)= - 2 \ne 0.
\]
In other words,  $\{\nu_{0}^+, \nu_{0}^-, \nu_{k}^+, \nu_{k}^- \
:  \ k\in \Z^* \}$ forms a Riesz basis for $X_s$ with
\[
\nu_{k}^{\pm}= \frac{\eta_{k}^{\pm}}{\|\eta_{k}^{\pm}\|_{X_s}},
\]
Moreover, we also have that $\nu^{(1),\pm}_{k} = b_{k}^{\pm} e^{ikx}$ with
\begin{equation}\label{bk}
0<C_1\leq |b_{k}^{\pm} |\leq C_2, \ \ k\in \Z.
\end{equation}

From the above discussion we have the following result.
\begin{theorem}
Let $\lambda_k^{\pm}={\pm} i \mu_k $ and $\phi_{ k}^{\pm} $ be given by
\[
\mu_k=  \sqrt{\frac{k^2}{1+k^2}}= \frac{|k|}{\sqrt{1+k^2}}, \ \ \  k\in \Z,
\]
\[
\phi_{k}^{+} = \left\{
\begin{array}{rl} \nu_{k}^{+},   & k=0, 1, 2, 3,\ldots\\
\nu_{ k}^{-},        & k=-1, -2, -3,\ldots,
\end{array} \right. \ \ \ \ \phi_{k}^- = \left\{
\begin{array}{rl} \nu_{-k}^+,   & k= 1, 2, 3,\ldots\\
\nu_{ -k}^-,        & k= 0, -1, -2, -3,\ldots,
\end{array} \right.
\]
then we have that

\medskip

\nd a) The spectrum of the operator A is $\sigma(A)=\{\lambda_k^{\pm} \  :
\ k\in \Z\}$, in which each $\lambda_k^{\pm} $ is a double eigenvalue with
eigenvectors $\phi_{k}^+$ and $\phi_{k}^-$ for $k\in \Z$.

\medskip

\nd b) The set $\{ \phi_{k}^{\pm},  \  :  \  k\in \Z \}$
forms an orthonormal basis for the space $X^s$. Thus, any $U\in
X^s$ has the following Fourier expansion

\smallskip

\[
U= \sum_{k\in \Z} \left(\alpha_{k }^+ \phi_{k}^+ + \alpha_{k }^-
\phi_{k}^-  \right),  \  \  \alpha_{k}^{\pm}=\left<U, \phi_{k}^{\pm}\right>_{X^s}, \  \ k\in \Z.
\]
\end{theorem}

\subsection{\bf Moving distributed controllability of the Linear System }
In this subsection, we consider the exact controllability problem for the linear system associated to \eqref{moving1}:
\begin{equation}\label{lp}
\begin{cases}
\partial_{t}u-\partial_x^2\partial_tu-\partial_xv=b(x+ct)h(x,t),  & \ \ \ x\in \T,\ \  t\in [0,T], \\
\partial_{t}v-\partial_xu=0, & \ \ \ x\in \T,\ \  t\in  [0,T], \\
u(x,0)=u_0(x), \,\, v(x,0)=v_0(x), & \ \ \ x\in \T.
\end{cases}
\end{equation}

Let us first deduce a necessary and sufficient condition for the exact controllability 
property of \eqref{lp} holds. By $\left\langle \cdot, \cdot \right\rangle$, we denote the duality product between $X^s$ and its dual $X^{-s}$. 
For $(\varphi_T, \psi_T) \in X^{-s}$, consider the following backward homogeneous equation
\begin{equation}\label{lp2'}
\begin{cases}
\partial_{t} \varphi-\partial_x^2\partial_t \varphi-\partial_x \psi=0,  & \ \ \ x\in \T,\ \  t\in \R, \\
\partial_{t}\psi-\partial_x \varphi=0, & \ \ \ x\in \T,\ \  t\in \R, \\
\varphi(x,T)=\varphi_T(x), \,\, \psi(x,T)=\psi_T(x), & \ \ \ x\in \T.
\end{cases}
\end{equation}
Let $(\varphi,\psi) \in  C([0, T],X^{-s})$ be the unique weak solution of \eqref{lp2'}.
\begin{lemma}\label{lemmacontrol}
Let $b=b(x) \in C^{\infty}(\T)$, be such that
\begin{equation}
\omega =\left\lbrace x \in \T: b(x)\neq 0 \right\rbrace \neq \emptyset.
\end{equation}
Then, the control $h(x,t) \in L^2\left(0,T; H^{s-2}(\T)\right) $ drives the inital data $(u_0,v_0) \in X^s$ of the system \eqref{lp} to $(u_T,v_T) \in X^s$ in time $T$ if and only if 
\begin{multline}\label{controlcondition}
  \int_0^T \int_{\T} h(x,t)b(x+ct)\varphi(x,t)dxdt=\left\langle u_T,\varphi_T-\partial_x^2 \varphi_T \right\rangle +  \left\langle v_T, \psi_T \right\rangle  \\
  - \left\langle u_0,\varphi(0)   +  \partial_x^2 \varphi(0) \right\rangle -  \left\langle v_0, \psi(0) \right\rangle
\end{multline}
for all $(\varphi_T,\psi_T) \in X^{-s}$, where $(\varphi, \psi) $ is the corresponding solution of \eqref{lp2'}.
\end{lemma}
\begin{proof}
Let us first suppose that $(u_0, v_0)$ and $(\varphi_T,\psi_T)$ belong to $\mathcal{D}(\T) \times \mathcal{D}(\T)$ and $h \in \mathcal{ D}((0,T)\times \T )$ and let $(u,v)$ and $(\varphi,\psi)$ be the (regular) solutions of \eqref{lp} and \eqref{lp2'}. We recall that $\mathcal{D}(M)$ denotes the set of $C^{\infty}(M)$ functions with compact support in $M$.
By multiplying the equation of $(u,v)$ by $(\varphi,\psi)$ and by integrating by parts one obtains, 
\begin{multline}\label{controlcondition1}
  \int_0^T \int_{\T} h(x,t)b(x+ct)\varphi(x,t)dxdt=\int_{\T}u(x,T) (\varphi_T(x)-\partial_x^2\varphi_T(x))dx +  \int_{\T}  v(x,T) \psi_T dx  \\
  -  \int_{\T} u_0(x)(\varphi(x,0)   +  \partial_x^2 \varphi(x,0))dx  -  \int_{\T}  v_0(x) \psi(x,0)dx.
\end{multline}
From a density argument we deduce, by passing to the limit in \eqref{controlcondition1}, that for any $(u_0,v_0) \in X^s$ and $(\varphi_T,\psi_T) \in X^{-s}$, 
\begin{multline}\label{controlcondition2}
  \int_0^T \int_{\T} h(x,t)b(x+ct)\varphi(x,t)dxdt=\left\langle u(T),\varphi_T-\partial_x^2 \varphi_T \right\rangle +  \left\langle v(T), \psi_T \right\rangle  \\
  - \left\langle u_0,\varphi(0)   +  \partial_x^2 \varphi(0) \right\rangle -  \left\langle v_0, \psi(0) \right\rangle.
\end{multline}
Now, from \eqref{controlcondition2}, it follows immediately that \eqref{controlcondition} holds if and only if $(u_0, v_0)$
is controllable to $(u_T,v_T)$  and $h$ is the corresponding control. This completes the proof.
\end{proof}
\begin{remark}
Without any loss of generality, we shall consider only the case $u_0 = v_0 = 0$. Indeed, let $(u_0,v_0)$, $(u_T,v_T)$ in $X^s$ and $h$ in $L^2(0,T;H^{s-2}(\T))$ be  the control which lead the solution $(\widetilde{u}, \widetilde{v})$ of \eqref{lp} from the zero initial data to the final state $(u^T,v^T)-(u(T),v(T))$, where $(u,v)$ is the mild solution corresponding to \eqref{lp} with initial data $(u_0,v_0)$. It follows immediately that these controls also lead to the solution $(\widetilde{u},\widetilde{v})+(u,v)$ of \eqref{lp} from $(u_0,v_0)$ to the final state $(u_T,v_T)$.
\end{remark}

We have the following exact controllability result.

\begin{theorem}\label{l-pro}
Let $T>0$, $|c|>2$ and $s\geq 0$ be given. Let $b=b(x) \in C^{\infty}(\T)$ be such that
\begin{equation}
\omega =\left\lbrace x \in \T: b(x)\neq 0 \right\rbrace \neq \emptyset.
\end{equation}
Then, for all $T > \frac{2\pi}{\Delta}$ with  
\begin{align*}
\Delta = \liminf_{k \rightarrow \infty} \left( \mu_{k+1}^{\pm}-\mu_k^{\pm}\right), \quad \mu_k^{\pm} = ck \pm \frac{|k|}{\sqrt{1+k^2}},
\end{align*} 
and  for any $ (u_T, v_T)\in X^s$, 
there exists a function $ h\in L^2\left(0,T; H^{s-2}(\T)\right)$ such that  the system 	\eqref{lp} with null initial data admits a unique solution $(u,v) \in C\left([0,T], X^{s}\right)$ satisfying
\begin{equation}
  (u(x, 0), v(x,0))=(0,0)  \quad \text{and} \quad  (u(x, T),v(x,T))= (u_T(x), v_T(x)).
\end{equation}
Moreover, there exists a constant $C > 0$ depending only on $s$ and $T$ such that
\begin{equation}\label{dependcontrol}
\|h\|_{L^2\left(0,T; H^{s-2}(\T)\right)} \leq C  \| (u_T, v_T)\|_{X^s}
\end{equation}
\end{theorem}

\begin{proof}
Firstly, note that the adjoin system associated to \eqref{lp} is given by 
\begin{equation}\label{lp2}
\begin{cases}
\partial_{t} \varphi-\partial_x^2\partial_t \varphi-\partial_x \psi=0,  & \ \ \ x\in \T,\ \  t\in \R, \\
\partial_{t}\psi-\partial_x \varphi=0, & \ \ \ x\in \T,\ \  t\in \R.
\end{cases}
\end{equation}
with final data

Note that if $(u,v)$ is the solution of \eqref{lp} with $h=0$, then $\varphi(x,t)=u(2\pi - x, T-t)$ and $\psi(x,t)=v(2\pi - x, T-t)$ is the solution of adjoin system \eqref{lp2}.  Thus, from the semigroup approach \eqref{solutionlinear} the solution takes the form
\begin{equation}\begin{cases}\label{cara}
\varphi(x,t)=\sum_{k \in \Z} \left( \alpha_k \cos \left( \sqrt{\frac{k^2}{1+k^2}}(T-t)\right) + \dfrac{i}{\sqrt{1+k^2}}\beta_k \sin \left( \sqrt{\frac{k^2}{1+k^2}}(T-t)\right) \right) e^{-ikx} \\
\\
\psi(x,t)=\sum_{k \in \Z} \left( \beta_k \cos \left( \sqrt{\frac{k^2}{1+k^2}}(T-t)\right) + i\sqrt{1+k^2} \alpha_k \sin \left( \sqrt{\frac{k^2}{1+k^2}}(T-t)\right) \right) e^{-ikx}. 
\end{cases}
\end{equation}

Let us multiply the first equation of the system \eqref{lp} by $\varphi$, the second one by $\psi$ and integrating by parts on $\T \times [0,T]$, it follows that 
\begin{equation}\label{new1}
 \int_{\T} \left[u\varphi     - u \partial_x^2 \varphi +   v \psi \right]_0^Tdx    =   \int_0^T \int_{\T} h(x,t)b(x+ct)\varphi(x,t)dxdt.
\end{equation}
Now, taking $k \in \Z$,  if $\varphi^{\pm}_k(x,t)= e^{\pm i  \sqrt{\frac{k^2}{1+k^2}}(T-t)}e^{-ikx}$, then $\psi^{\pm}_k(x,t)= \pm\sqrt{1+k^2} \varphi^{\pm}_k(x,t)$ or if $\psi^{\pm}_k(x,t)= e^{\pm i  \sqrt{\frac{k^2}{1+k^2}}(T-t)}e^{-ikx}$, then $\varphi^{\pm}_k(x,t)=\pm \frac{1}{\sqrt{1+k^2}} \psi^{\pm}_k(x,t)$. It is easy to see that any case we get the same moment problem. Thus, from \eqref{new1}, it follows that
\begin{multline*}
 \int_0^T \int_{\T} h(x,t)b(x+ct)e^{\pm i  \sqrt{\frac{k^2}{1+k^2}}(T-t)}e^{-ikx}dxdt \\
 = e^{\pm i  \sqrt{\frac{k^2}{1+k^2}}T} \left\lbrace (1+k^2) \int_{\T}   e^{-ikx} \left[   e^{\mp i  \sqrt{\frac{k^2}{1+k^2}}T} u(x,T)   \right]dx  \right. \\
 \left.\pm    \sqrt{1+k^2} \int_{\T}   e^{-ikx} \left[ e^{\mp i  \sqrt{\frac{k^2}{1+k^2}}T}    v(x,T)   \right]dx\right\rbrace.
\end{multline*}
By a simple change of variables, we obtain
\begin{align*}
\int_0^T \int_{\T} h(x,t)b(x+ct)e^{\pm i  \sqrt{\frac{k^2}{1+k^2}}(T-t)}e^{-ikx}dxdt=\int_0^T \int_{\T} \widetilde{h}(x,t)b(x)e^{\pm i    \sqrt{\frac{k^2}{1+k^2}}(T-t)}e^{-ik(x-ct)}dxdt,
\end{align*}
where $\widetilde{h}(x,t)=h(x-ct,t)$.  Then, the moment problem consists in finding a control $\widetilde{h}$ such that for all $k \in \Z$,
\begin{multline}\label{moment}
\int_0^T \int_{\T} \widetilde{h}(x,t)b(x)e^{ i \left( kc \mp     \sqrt{\frac{k^2}{1+k^2}}\right) t}e^{-ikx}dxdt=  (1+k^2)  \int_{\T}   e^{-ikx} \left[   e^{\mp i  \sqrt{\frac{k^2}{1+k^2}}T} u(x,T)  \right]dx   \\
 \pm \sqrt{1+k^2}  \int_{\T}   e^{-ikx} \left[ e^{\mp i  \sqrt{\frac{k^2}{1+k^2}}T}    v(x,T)   \right]dx.
\end{multline}

Now, we choose $\left\{q_{m}^{\pm}\right\}_{m \in \mathbb{Z}^{*}} \subset L^{2}(0, T) $  as a biorthogonal family to the set $$ S=\left\{ e^{i(k c \mp \sqrt{\frac{k^{2}}{k^{2}+1}}) t}\right\}_{k \in  \Z} .$$
It means that\textcolor{blue}{, }
\begin{align*}
\int_0^T q_{m}^{\pm}(t) e^{i\left( k c \mp \sqrt{\frac{k^{2}}{k^{2}+1}}\right) t} dt = \delta_{km}\textcolor{blue}{, } \quad \text{and} \quad \int_0^T q_{m}^{\pm}(t) e^{i\left( k c \pm \sqrt{\frac{k^{2}}{k^{2}+1}}\right) t} dt = 0,  \quad \textcolor{blue}{\forall } m,k \in \Z,
\end{align*}
where $\delta_{km}$ is Kronecker's delta. The existence of this biorthogonal family can be established as explained in \cite[Remark 16]{cerpa2018} under the hypoteshis $|c| > 2$ and  $T > \frac{2\pi}{\Delta}$. The condition $|c|>2$ guaranties that $\Delta$ is positive and to avoid the existence of different $k$ and $m$ such that
\begin{align*}
ck +\frac{|k|}{\sqrt{1+k^2}}=cm -\frac{|m|}{\sqrt{1+m^2}}
\end{align*}
with it is necessary  to solve the moment problem with no additional compatibility conditions on the initial and final data,  \cite[Lemma 12]{cerpa2018}. We can cite \cite[section 2]{rusell1967} for more details.   Following \cite{rosier},  we look for a control  $\widetilde{h}(x,t)$ of the form
\begin{equation}\label{control1}
\widetilde{h}(x,t)=\sum_{m \in \Z}\left( f_m^+ q_m^+(t) + f_m^- q_m^-(t)\right)e^{imx},
\end{equation}
where  $\{q_m^{\pm} \}_{m\in \Z  }\in L^2(0,T)$ is a biorthogonal family of $S$ and the functions     $f_k^{\pm}$ are to be determined later to satisfy the moment problem \eqref{moment}. Thus, we obtain
\begin{equation}\label{control2}
f_k^+= \frac{e^{- i  \sqrt{\frac{k^2}{1+k^2}}T} \sqrt{1+k^2} }{\int_{\T}b(x)dx} \left[  \sqrt{1+k^2} \int_{\T}   e^{-ikx}    u_T(x)  dx   
 +  \int_{\T}   e^{-ikx}     v_T(x)  dx  \right]
\end{equation}
and
\begin{equation}\label{control3}
f_k^-= \frac{e^{ i  \sqrt{\frac{k^2}{1+k^2}}T} \sqrt{1+k^2} }{\int_{\T}b(x)dx} \left[  \sqrt{1+k^2} \int_{\T}   e^{-ikx}    u_T(x)  dx   
 -  \int_{\T}   e^{-ikx}     v_T(x)  dx  \right]
\end{equation}
where $\int_{\T}b(x)dx\neq 0$ by hypothesis.   Now, The result will be proved if we prove that the control function $h$ belong to  $L^2(0, T; H^{s-2}(\T) )$.  From now on, we will denote by $C > 0$ a general constant which may vary from line to line. Note that 
\begin{equation*}
\begin{aligned}
\|\tilde{h}\|^2_{L^{2}\left(0, T ; H^{s-2}(\T)\right)} &=\int_{0}^{T}\left\|\sum_{j \in \mathbb{Z}}\left(f_{j}^{+} q_{j}^{+}(t)+f_{j}^{-} q_{j}^{-}(t)\right) e^{i j x}\right\|_{H^{s-2}(\T)}^{2} d t \\
 & \leq C \int_{0}^{T}   \sum_{j \in \mathbb{Z}} \left(1+j^{2}\right)^{s-2}\left|f_{j}^{+} q_{j}^{+}(t)+f_{j}^{-} q_{j}^{-}(t)\right|^{2} d t \\ 
 & \leq C\left(    \left\|u_{T}\right\|_{H^{s}(\T)}^{2}+ \left\|v_{T}\right\|_{H^{s-1}(\T)}^{2}\right). 
\end{aligned}
\end{equation*}
Thus, from Lemma \ref{lemmacontrol} the function $h(x, t) = \widetilde{h}(x + ct, t)$ given by \eqref{control1} - \eqref{control3} is the desired control  that drives the system from $(0, 0)$ to $(u_T, v_T)$.
\end{proof}

\subsection*{Nonlinear System}

We turn to the nonlinear system and we prove the exact controllability for the full system¡. Firstly, we introduce the linear (bounded) operator 
\begin{equation}
\Phi: X^s \times X^s \longrightarrow  L^2(0, T ; H^{s-2}(\T))
\end{equation}
defined by
\begin{equation}
\Phi \left( (u_0,v_0), (u_T,v_T)\right)(t) = h(t),
\end{equation}
where $h$ is given by \eqref{control1} and $f_m^{\pm}$ is the solution of \eqref{control2}-\eqref{control3} with $( (\hat{u}_0)_k,(\hat{v}_0)_k)$ and $( (\hat{u}_T)_k,(\hat{v}_T)_k)$ substituted to $( (\hat{u}_k)(0),(\hat{v}_k)(0))$ and $( (\hat{u}_k)(T),(\hat{v}_k)(T))$, respectively. Thus,  $h=\Phi \left( (u_0,v_0), (u_T,v_T)\right)$ is a control driving the solution $(u,v)$ of \eqref{lp} from $ (u_0,v_0)$ at $t = 0$ to $ (u_T,v_T)$ at $t = T$.

\subsection*{\bf Proof of Theorem \ref{mainmovilcontrol}}
Pick any time $T > 2\pi/\Delta$, and any $ (u_0,v_0), (u_T,v_T) \in X^s$ (If $p=1$, $s=1$ and if $p>1$, $s>1/2$) satisfying 
\begin{equation}
\|(u_0, v_0)\|_{X^s}+ \| (u_T, v_T)\|_{X^s} < \delta, 
\end{equation}
with $\delta$ to be determined later.  Recall that we can rewrite (\ref{lp}) as the  first order evolution system
\begin{equation}\label{nlpGen}
\begin{cases}
U_t=AU-G(U)+Bh, \\  \\
U(0, \cdot)=(u(0,\cdot), v(0,\cdot))=(u_0,v_0)=U_0,
\end{cases}
\end{equation}
where 
\[
 G(U) =\begin{pmatrix} \left(I-\partial_x^2\right)^{-1}\partial_x\left( \frac1{p+1} u^{p+1}\right)\\ \\
                    0\end{pmatrix},  \   \  Bh=\begin{pmatrix}  \left(I-\partial_x^2\right)^{-1}(\,\tilde{b} h\,) \\  \\ 0
\end{pmatrix},
\]
with $\tilde{b}(x)=b(x+ct)$. Consider the operator
\begin{equation*}\label{w1}
w(U, T)= \int_0^T S(T-\tau)G(U)\,d\tau.
\end{equation*}
 Note that Lemma \ref{pG}  implies that 
\begin{equation}
\|w(U,T)-w(V,T)\|_{X^s}\leq C_1T\left(\|U\|_{L^{\infty}(0,T;X^s)}+\|V\|_{L^{\infty}(0,T;X^s)}\right)^p\|U-V\|_{L^{\infty}(0,T;X^s)}
\end{equation}
for some positive constant $C_1$. We are led to consider the nonlinear map
\begin{equation}
\Psi(U)= S(t)U_0 - \int_0^t S(t-\tau)G(U)(\tau)\,d\tau+ \int_0^t S(t-\tau)\bigl[B\Phi(U_0, U_T+ w(U, T))\bigr](\tau)\,d\tau.
\end{equation}
Note that 
\begin{equation}
\Psi(U)(t)=
\begin{cases}
U_0 & \text{if $t=0$,} \\
-w(U, T) + \left( U_T+ w(U, T)\right) =U_T & \text{if $t=T$.}
\end{cases}
\end{equation}
Thus, in order to obtain the exact controllability result of the nonlinear system  we need to prove that the operator $\Psi(U)$ is a contraction in a appropriate space.  For any $R>0$, let $$B_R(T)=\left\{U\in L^\infty(0,T;X^s)\,\,:\,\,\|U\|_{L^\infty(0,T;X^s)}\leq R\right\}.$$ 
So, it follows that
\begin{align*}
\|\Psi(U(t))\|_{X^s}&\leq C_2\left(\|U_0\|_{X^s}+\|\Phi(U_0, U_T+ w(U, T))\|_{L^1\left(0,T, \,H^{s-2}\right)}\right)+C_1R^{p+1}T, \quad U 	\in B_R(T),
\end{align*}
for some positive constant $C_2$. From \eqref{dependcontrol}, we deduce that
\begin{align*}
\|\Psi(U(t))\|_{X^s}&\leq C_2\left(\|U_0\|_{X^s}+\|U_T\|_{X^s}+ \|w(U, T)\|_{X^s}\right)+R^{p+1}TC_1,  \\
&\leq C_2 \left(\|U_0\|_{X^s}+\|U_T\|_{X^s}\right)+ R^{p+1}T(C_2+1)C_1, \\
&\leq C_2 \delta+ R^{p+1}T(C_2+1)C_1.
\end{align*}
On the other hand, note that
\begin{multline}
\Psi(U(t))- \Psi(V(t))= \int_0^t S(t-\tau)(G(V)(\tau)-G(U)(\tau))\,d\tau\\
+ \int_0^t S(t-\tau)\bigl[B\Phi(0, w(U, T)-w(V, T))\bigr](\tau)\,d\tau.  
\end{multline}
Hence, there exists a positive constant $C_3=C_3(T)$ such that
\begin{align*}
\|\Psi(U(t))- \Psi(V(t))\|_{X^s} &\leq C_2 \left( \|w(V, T) -w(U, T) \|_{X^s} + \|\Phi(U_0, w(U, T)-w(V, T))\|_{L^1([0,T; \,H^{s-2})}
 \right) \\
 &\leq2 C_3\|w(V, T) -w(U, T) \|_{X^s}\\
 &\leq 2C_1C_3T\left(\|U\|_{L^{\infty}(0,T;X^s)}+\|V\|_{L^{\infty}(0,T;X^s)}\right)^p\|U-V\|_{L^{\infty}(0,T;X^s)}  \\
 &\leq 2^{p+1}C_1C_3 TR^p\|U-V\|_{L^{\infty}(0,T;X^s)}.
\end{align*}
Picking $R= \min\left\lbrace \left[ \dfrac{1}{2^{p+2}C_1C_3 T} \right]^{1/p}, \left[ \dfrac{1}{2T(C_2+1)C_1} \right]^{1/p}	\right\rbrace$ and $\delta = \dfrac{R}{2C_2}$,  we obtain for $U_0$, $U_T \in X^s$ satisfying
\begin{equation*}
\|U_0\|_{X^s}+\|U_T\|_{X^s}< \delta
\end{equation*}
and $U,V \in B_R(T)$ such that
\begin{align*}
\|\Psi(U(t))\|_{L^\infty(0,T;X^s)} &\leq R, \\
\|\Psi(U(t))- \Psi(V(t))\|_{L^\infty(0,T;X^s)} &\leq\frac{1}{2} \|U-V\|_{L^\infty(0,T;X^s)}.
\end{align*}
It follows from the contraction mapping theorem that $\Psi$ has a unique fixed point $U$ in $B_R(T)$. Then $U$ satisfies (\ref{moving1})–(\ref{contcond}) with $h = \Phi(U_0, U_T - w(U,T))$ and $U(T ) = U_T$, as desired. The proof of Theorem \ref{mainmovilcontrol} is complete.

$\hfill\square$

\subsection{Point Control for the Linear System}
In this subsection, we turn to the pointwise contolabillity problem of the linear system \eqref{moving2}. As was done with the exact controllability of \eqref{moving1}, we could establish a necessary and sufficient condition to obtain the controllability result. 

\begin{lemma}\label{lemmacontrol_1}
There exist  control functions $g \in L^2(0,T)$ drives the inital data $(u_0,v_0) \in H^{-1}(\T) \times L^2(\T)$ to $(u_T,v_T) \in X^s$ in time $T$ of the systems \eqref{moving2}   if and only if 
\begin{multline*}
  \int_0^T  g(t)\varphi(-ct,t)dt=\left\langle u_T, \varphi_T \right\rangle_{H^{-1}(\T), H^1(\T)}     +  (  v_T, \psi_T )_{L^2(\T)} \\-  \left\langle u_0, \varphi(0) \right\rangle_{H^{-1}(\T), H^1(\T)}     -  (  v_0, \psi(0) )_{L^2(\T)}  
\end{multline*}
for all $(\varphi_T,\psi_T) \in H^{1}(\T)\times L^2(\T)$, where $(\varphi, \psi) $ is the corresponding solution of \eqref{lp2'}.
\end{lemma}

Thus, we can present the proof of Theorem \ref{mainmovilcontrol2}:

\begin{proof}[\textbf{Proof of Theorem \ref{mainmovilcontrol2}}]

The general identity \eqref{new1} can be applied to the case of system \eqref{moving2}, taking the form 
\begin{equation}
 \int_{\T} \left[u\varphi     - u \partial_x^2 \varphi +   v \psi \right]_0^Tdx    =   \int_0^T g_1(t) \left\langle \delta(x+ct), \varphi(t)\right\rangle dt=\int_0^T g_1(t) \varphi(-ct,t) dt
\end{equation}
From \eqref{cara}, if follows that 
\begin{multline}\label{deltaprima1}
\int_0^T  g(t)   e^{ i\left(ck\mp\sqrt{\frac{k^2}{1+k^2}}\right)t}dt=  (1+k^2) \int_{\T}   e^{-ikx} \left[   e^{\mp i  \sqrt{\frac{k^2}{1+k^2}}T} u(x,T) -  u_0(x)  \right]dx   \\
 \pm    \sqrt{1+k^2} \int_{\T}   e^{-ikx} \left[ e^{\mp i  \sqrt{\frac{k^2}{1+k^2}}T}    v(x,T) -  v_0(x)  \right]dx
\end{multline}
Thus, the result will be proved if we can construct a control function $g \in L^2(0, T)$ fulfilling \eqref{deltaprima1}. In fact,  we consider the same biorthogonal family as in the proof of Theorem \ref{l-pro} and the function 
\begin{equation}\label{controlpuntual}
g(t)=\sum_{m \in \Z} \left\lbrace \alpha_m^+ q_m^+(t) +  \alpha_m^-  q_m^-(t)\right\rbrace \in L^2(0,T).
\end{equation}
where the coefficients $\alpha^+_m$,  $\alpha^-_m$ are given by 
\begin{multline*}
\alpha^+_m =  (1+k^2) \int_{\T}   e^{-ikx} \left[   e^{- i  \sqrt{\frac{k^2}{1+k^2}}T} u_T(x) -  u_0(x)  \right]dx   \\
 +   \sqrt{1+k^2} \int_{\T}   e^{-ikx} \left[ e^{- i  \sqrt{\frac{k^2}{1+k^2}}T}    v_T(x) -  v_0(x)  \right]dx 
\end{multline*}
\begin{multline*}
\alpha^-_m =  (1+k^2) \int_{\T}   e^{-ikx} \left[   e^{i  \sqrt{\frac{k^2}{1+k^2}}T} u_T(x) -  u_0(x)  \right]dx   \\
 -   \sqrt{1+k^2} \int_{\T}   e^{-ikx} \left[ e^{ i  \sqrt{\frac{k^2}{1+k^2}}T}    v_T(x) -  v_0(x)  \right]dx 
\end{multline*}
Hence, from Lemma \ref{lemmacontrol_1} the function $g(t)$ given by \eqref{controlpuntual}  is the desired control  that drives the system from $(u_0, v_0)$ to $(u_T, v_T)$.

\end{proof}


\section{Another Controllability Aspects on a Bounded Interval }

In this section, we are concerned with the of the linearized symmetric regularized long wave system  posed on the finite interval $[0, 1]$. Given a time $T > 0$, we consider an initial condition $(u_0,v_0)$ and a target condition $(u_T,v_T)$ on an appropriate space. We would like to found  a control functions $h = h(t)$ also on an adequate space,  such that the solution of the problem \eqref{linear1} with the Direchlet - Neumann boundary \eqref{linear2} has some controllability property.

\subsubsection{\textbf{The homogeneous System}}

We establish the wellposedness of the nonhomegenous SRLW system 

\begin{equation}\label{hs1}
\begin{cases}
\partial_{t}u-\partial_t\partial_x^2u-\partial_xv=0, & (x,t) \in (0,1)\times (0,T), \\ 
\partial_{t}v-\partial_xu=0,  & (x,t) \in (0,1)\times (0,T),\\
u(0,t)=u(1,t)=\partial_xv(0,t)=\partial_xv(1,t)=0, & t \in (0,T), \\
u(x,0)=u_0(x), 	\quad v(x,0)=v_0(x), & x \in (0,1).
\end{cases}
\end{equation}
We will establish the wellposedness results by using spectral methods. Indeed, we can rewrite the homogeneous system	\eqref{hs1} as follows,
\begin{equation}\label{A1}
\begin{cases}
U_t +AU=0, \\
U(0)=U_0,
\end{cases}
\end{equation}
where $U=(u,v)^T$, $U_0=(u_0,v_0)^T$ and for $$D(A)=\left\lbrace (u,v) \in  H^1_0(0,1)\times  H^1(0,1): \partial_xv(0)=\partial_x v(1)=0 \right\rbrace,$$ we can define the operator 
\begin{equation}
A:D(A)	\subset H^1_0(0,1)\times L^2(0,1) \longrightarrow  H^1_0(0,1)\times L^2(0,1),
\end{equation}
given by 
\begin{equation}\label{operatorA}
A=\begin{pmatrix}0 & -\left(I-\partial_x^2\right)^{-1}\partial_{x}\\ \\
-\partial_{x} & 0
       \end{pmatrix}, 
\end{equation}

\begin{proposition}
A is a compact, skew-adjoint operator in $H^1_0(0,1)\times L^2(0,1)$.
\end{proposition}

\begin{proof}
Firstly, note that the operator $(I-\partial_x^2)^{-1}\partial_x: H^1_0 (0, 1) \rightarrow H^1_0 (0, 1)$ and $\partial_x: H^1_0 (0, 1) \rightarrow L^2 (0, 1)$. It is easy to see that the operator  $\partial_x$ is compact in $L^2(0,1)$. Due to the regularizing effect of the operator $(I-\partial_x^2)^{-1}$, it follows that  $(I-\partial_x^2)^{-1}\partial_x$ takes values in $H^2(0, 1) \cap H^1_0 (0, 1)$ which is compactly embedded in $H^1_0 (0, 1)$. Hence $(I-\partial_x^2)^{-1}\partial_x$ is compact, thus $A$ is compact as well.  

Consider in $H^1_0(0,1) \times L^2(0,1)$ the inner product given by
\begin{equation}
\left( (\omega,\eta), (\varphi,\psi)^{T}  \right)_{H^1_0\times L^2} = \left( \omega,\varphi\right)_{H^1_0}+\left(v,\psi\right)_{L^2} = \int_0^1\left( \omega_x\varphi_x + \omega\varphi\right)dx+\int_0^1 v\psi dx.
\end{equation}
For any $(u,v)^{T}, (\varphi,\psi)^{T}$ in $   \left[ H^1_0(0,1)\cap H^2(0,1) \right] \times V    $, where 
$$V=\left\lbrace v \in H^1_0(0,1)\cap H^2(0,1): \partial_xv(0)=\partial_x v(1)=0\right\rbrace,$$
we have that
\begin{align*}
\left( A(u,v)^{T}, \right.& \left.(\varphi,\psi) \right)_{H^1_0\times L^2} = \left( (-\left(I-\partial_x^2\right)^{-1}\partial_{x}v,-\partial_x u)^{T}, (\varphi,\psi) \right)_{H^1_0\times L^2} \\
= & \left( v, \partial_x \varphi \right)_{L^2}-\int_0^1  \partial_x u  \left(I-\partial_x^2\right)^{-1}\left(I-\partial_x^2\right)\psi dx, 
\end{align*}
therefore, 

\begin{align*}
\left( A(u,v)^{T}, \right.& \left.(\varphi,\psi) \right)_{H^1_0\times L^2} =  \left( v, \partial_x \varphi \right)_{L^2}-\int_0^1  \partial_x u  \left(I-\partial_x^2\right)^{-1}\psi dx   +\int_0^1  \partial_x u  \left(I-\partial_x^2\right)^{-1}\partial_x^2\psi dx    \\
= & \left( u , \left(I-\partial_x^2\right)^{-1}\partial_x \psi \right)_{H^1_0}+\left( v, \partial_x \varphi \right)_{L^2}  \\
&= \left( (u,v)^{T} ,( \left(I-\partial_x^2\right)^{-1}\partial_x \psi, \partial_x \varphi) \right)_{H^1_0\times L^2} = \left( (u,v)^{T} , (-A)(   \varphi,\psi) \right)_{H^1_0\times L^2} 
\end{align*}
By density, it follows that $\left( A(u,v)^{T}, (\varphi,\psi)\right)_{H^1_0\times L^2} = \left( (u,v) , (-A)( \varphi,\psi)^{T} \right)_{H^1_0\times L^2}$, for all $u, \varphi \in H^1_0(0,1)$ and $v, \psi \in L^2(0,1)$. 

\end{proof}

From above proposition, \eqref{A1} can be treated like an ordinary differential equation in the Hilbert space $H^1_0(0,1)\times L^2(0,1)$. By using Cauchy–Lipschitz–Picard \cite[pag 104]{brezis} and since $A$ is skew-adjoint, the following result holds immediately.

\begin{proposition}
The equation \eqref{A1} has a unique solution $U \in C^1([0,\infty);H^1_0(0,1)\times L^2(0,1))$. Moreover 
\begin{multline}\label{ener1}
\int_0^1|u(x,t)|^2dx + \int_0^1|\partial_x u(x,t)|^2dx + \int_0^1|v(x,t)|^2dx \\
= \int_0^1|u_0(x)|^2dx + \int_0^1|\partial_x u_0(x)|^2dx + \int_0^1|v_0(x)|^2dx 
\end{multline}
\end{proposition}
\begin{proof}
The existence and uniqueness follow from  Cauchy–Lipschitz–Picard theorem  \cite[pag 104]{brezis}. To see \eqref{ener1}, note that the $H^1_0(0,1)\times L^2(0,1)$ norm is conserved. Indeed,
\begin{align*}
\frac{1}{2}\frac{\partial}{\partial t}\|(u,v)\|_{H^1_0\times L^2}^2=\frac{1}{2}\frac{\partial}{\partial t}\left(\|u\|_{H^1_0}^2+\|v\|_{L^2}^2\right)=Re \left( (u,v), (-A)(u,v)^{T}\right)_{H^1_0\times L^2}=0.
\end{align*}
\end{proof}

We need to analyze the spectral decomposition of the operator $A$. The next proposition give us the behavior of the eigenvalues and eigenfunction of operator $A$.

\begin{proposition}
$A$ has a sequence of purely imaginary eigenvalues $(\lambda_n)_{n\in \Z^*}$,
\begin{equation}
\lambda_n=\sgn(n)\frac{n\pi i}{\sqrt{1+n^2\pi^2}}, \quad n\in \Z^*.
\end{equation}
Moreover, to each eigenvalue $\lambda_n$  corresponds an unique eigenfunction $U_n(x)=(u_n(x),v_n(x))^T,$ where 
\begin{equation}
u_n(x)=\frac{i\sqrt{2}}{2\sqrt{1+n^2\pi^2}}\sin(n\pi x), \quad \text{and} \quad v_n(x)= \frac{\sqrt{2}\sgn(n)}{2}\cos (n\pi x)  .
\end{equation}
 The family $(U_n)_{n\in \Z^*}$ forms an orthonormal basis in $H^1_0(0,1)\times L^2(0,1)$.
\end{proposition}

\begin{proof}
We are looking for $\lambda \in \C$ and $U=(u,v)^T \in H^1_0(0,1)\times L^2(0,1)$, such that $AU=\lambda U$, which is equivalent to 
\begin{equation}\label{A3}
\begin{cases}
(I-\partial_x^2)^{-1}\partial_x v=  \lambda u, \\ 
\partial_xu=\lambda v, \\
 u(0)=u(1)=\partial_x v(0)= \partial_x v(1)=0. \\
\end{cases}
\end{equation}
Thus, it yields
\begin{equation}\label{A2}
\begin{cases}
(1+\lambda^2)u_{xx}-  \lambda^2 u=0, \\ 
u(0)=u(1)=0. \\
\end{cases}
\end{equation}
Hence, $u(x)=c_1e^{\frac{|\lambda|}{\sqrt{1+\lambda^2}}x}+c_2e^{-\frac{|\lambda|}{\sqrt{1+\lambda^2}}x}.$
From the boundary condition, it follows that
\begin{equation}
\begin{cases}
c_1+c_2=0, \\ 
c_1e^{\frac{|\lambda|}{\sqrt{1+\lambda^2}}}+c_2e^{-\frac{|\lambda|}{\sqrt{1+\lambda^2}}}=0. \\
\end{cases}
\end{equation}
Then, $c_2=-c_1$ and $e^{\frac{2|\lambda|}{\sqrt{1+\lambda^2}}}=1$. Thus, we have a sequence $\{\lambda_n\}_{n\in \Z^*}$ such that 
\begin{equation}
\frac{|\lambda_n|}{\sqrt{1+\lambda_n^2}}=n\pi i, \quad n \in \Z^*.
\end{equation}
Therefore, 
\begin{equation}
\lambda_n= \sgn(n)\frac{n\pi i}{\sqrt{ 1+ n^2\pi^2}}, \quad n \in \Z^*.
\end{equation}
To each $\lambda_n$ associated to  \eqref{A2},  corresponds an eigenfunction $u_n$, where 
\begin{equation}\label{u1}
u_n(x)=2ic_1\sin(n\pi x).
\end{equation}
Coming back to system \eqref{A3}, we consider  $ v_{n}(x)=\frac{1}{\lambda_n} u_{n,x}(x)=\dfrac{2ic_1 n\pi }{\lambda_n}\cos(n\pi x)$. Hence, we have
\begin{equation}\label{v1}
v_n(x)=2c_1\sgn(n)\sqrt{ 1+ n^2\pi^2} \cos(n\pi x) \quad \text{and} \quad  v_{n,x}(x)=-2c_1n\pi\sgn(n)\sqrt{ 1+ n^2\pi^2} \sin(n\pi x),
\end{equation} 
implying that $v_{n,x}(0)=v_{n,x}(0)=0$. Thus, we have that $\{\lambda_n\}_{n\in \Z^*}$ forms the set of eigenvalues of $A$, whose eigenfunction are given by $U_n=(u_n,v_n)^T$, where $u_n$ and $v_n$ are given by \eqref{u1} and \eqref{v1}, respectively. Now, we choose the constant $c_1$ such that $\|(u_n,v_n)\|_{H^1_0\times L^2}=1$.  Indeed,  let us consider the norm of $U_n$ as $\|U_n\|_{H^1_0(0,1)\times L^2(0,1)}:=\|u_n\|_{H^1_0(0,1)}+\|v_n\|_{L^2(0,1)}.$ Note that 
\begin{align*}
\|\cos(n\pi x)\|_{L^2(0,1)}^2&=\int_0^1   \cos^2(2n\pi x) dx =  \left[\frac12x+\frac{\sin(4n\pi x)}{8n\pi} \right]_0^1 =\frac{1}{2} 
\end{align*}
and
\begin{align*}
\|\sin(n\pi x)\|_{H^1_0(0,1)}^2&=\|\sin(n\pi x)\|_{L^2(0,1)}^2+\|n\pi\cos(n\pi x)\|_{L^2(0,1)}^2 \\
&=  \left[\frac12x-\frac{\sin(4n\pi x)}{8n\pi} \right]_0^1+n^2\pi^2  \left[\frac12x+\frac{\sin(4n\pi x)}{8n\pi} \right]_0^1 =\frac{1+n^2\pi^2}{2}.
\end{align*}
Thus, we obtain that 
\begin{align*}
\|U_n\|_{H^1_0(0,1)\times L^2(0,1)}&=\|2ic_1\sin(n\pi x)\|_{H^1_0(0,1)}+\|2c_1\sgn(n)\sqrt{ 1+ n^2\pi^2} \cos(n\pi x)\|_{L^2(0,1)} \\
&=\frac{4}{\sqrt{2}}|c_1|\sqrt{ 1+ n^2\pi^2}
\end{align*}
Hence, choosing $c_1= \frac{\sqrt{2}}{ 4\sqrt{ 1+ n^2\pi^2} }$, we have that 
\begin{align*}
U_n(x)= \left( \frac{i\sqrt{2}}{2\sqrt{1+n^2\pi^2}}\sin(n\pi x),  \frac{\sqrt{2}\sgn(n)}{2}\cos (n\pi x)      \right) \quad \text{and} \quad \|U_n\|_{H^1_0\times L^2}=1.
\end{align*}
As $A$ is a compact and skew-adkjoint operator in  $H^1_0(0,1)\times H^1_0(0,1)$, the eigenfunction $\{U_n\}_{n\in \Z^*}$ forms an orthonormal basis in $H^1_0(0,1)\times H^1_0(0,1)$. 
\end{proof}

\begin{remark}
We have obtained that $\lim_{|n|\rightarrow \infty} |\mu_n| = 1$. This is due to the compactness of the operator $A$. Thus, the spectrum of $A$ admits a finite limit point and will have some very important consequences for the   controllability properties of the generalized SRLW equation. 
\end{remark}

Wiht the spectral analysis of $A$ and the asymptotic behavior of $\lambda_k$ in hands, we are able to obtain the explicit solution in the space $H^1_0(0,1)\times L^{2}(0,1)$. Indeed, if we consider a initial data $U_0 \in H^1_0(0,1)\times L^{2}(0,1)$, $U_0=\sum_{n\in \Z^*}a_nU_n$, the corresponding solution of equation \eqref{A1} is given by 
\begin{equation}
U(x,t)=\sum_{n\in \Z^*} a_nU_ne^{-i \lambda_n t}. 
\end{equation}

The next result is a directly consequence of the previous spectral analysis and ensures the  well-posedness of the system \eqref{hs1}.
\begin{proposition}\label{spectralsolution}
For any $U_0=(u_0,v_0) \in  H^1_0(0,1)\times L^{2}(0,1)$ initial data given by 
\begin{align*}
U_0=\sum_{n\in \Z^*}a_nU_n,
\end{align*}
for some constants $a_n$, where $\{U_n\}_{n\in \Z^*}$ is a orthonormal basis of $H^1_0(0,1)\times L^{2}(0,1)$ forms by the eigenfunctions of operator $A$. Then, there exists a unique solution of 
\begin{equation*}
\begin{cases}
U_t+AU=0, \\
U(0)=U_0,
\end{cases}
\end{equation*}
belongs to $C(\R,H^1_0(0,1)\times L^{2}(0,1))$ and it is given by 
\begin{equation}
U(x,t)=\sum_{n\in \Z^*} a_nU_ne^{-i \lambda_n t}
\end{equation}
where $U(x,t)=(u(x,t),v(x,t))^T$,
\begin{equation}
u(x,t)=\sum_{n\in \Z^*} \frac{i\sqrt{2}a_n}{2\sqrt{1+n^2\pi^2}} \sin(n\pi x)e^{-i \lambda_n t} \quad \text{and} \quad v(x,t)= \sum_{n\in \Z^*}  \frac{\sqrt{2}\sgn(n)a_n}{2}   \cos (n\pi x) e^{-i \lambda_n t},
\end{equation}
where $\lambda_n= \frac{ \sgn(n)n\pi }{\sqrt{ 1+ n^2\pi^2}},$ for $n \in \Z^*$.
\end{proposition}

\begin{remark}
From the cosine and sine functions, we see that there is no gain of regularity for the linear RSLW equation. Moreover, from the asymptotic behavior of
eingenvalues $\lambda_k$, we see that the fact that the position $(u,v)$ and the velocity $(u_t,v_t)$ have the same regularity. 
\end{remark}

\subsubsection{\textbf{The Nonhomogeneous System}}

Now, we are concerned  with the initial boundary value problem \eqref{linear1} - \eqref{linear2}.

\begin{theorem}\label{existence1}
Let $(u_0,v_0) \in H^1(0,1)\times L^2(0,1)$ and $h(t) \in H^1(0,T)$. Then, there exists a unique mild solution $(u,v)$ of problem \eqref{linear1}-\eqref{linear2} in $C([0,T];H^1(0,1)\times L^2(0,1))$.
\end{theorem}
\begin{proof}
If $(u,v)$ is solution of \eqref{linear1}-\eqref{linear2}, then $\varphi(x,t):=u(x,t)+xh(t)$ and $\psi(x,t):=v(x,t)$ are solutions of system,
\begin{equation}\label{linear3}
\begin{cases}
\partial_{t}\varphi-\partial_t \partial_x^2\varphi-\partial_x\psi=xh'(t), & (x,t) \in (0,1)\times (0,T), \\ 
 \partial_{t}\psi-\partial_x\varphi=-h(t),  & (x,t) \in (0,1)\times (0,T), \\
\varphi(0,t)=0,  \quad \varphi(1,t)=0, & t \in (0,T), \\
\partial_x \psi(0,t)=0,  \quad \partial_x \psi(1,t)=0, & t \in (0,T) \\
\varphi(x,0)=u_0(x)-xh(0), \quad \psi(x,0):=v_0(x), & x \in (0,1).
\end{cases}
\end{equation}
Considering $W=(\varphi,\psi)^{T}$,  the system  \eqref{linear3} takes the form 
\begin{equation}\label{linear4}
\begin{cases}
W_t+AW= F, \\
W(0)=W_0
\end{cases}
\end{equation}
where $A$ is given by \eqref{operatorA} and 
\begin{equation}\label{linear5}
F=\begin{pmatrix}\left(I-\partial_x^2\right)^{-1}(xh'(t))\\   \\
\left(I-\partial_x^2\right)^{-1}(-h(t))
       \end{pmatrix}, \quad \text{and} \quad W_0(x)=\begin{pmatrix}   u_0(x)-xh(0)\\   \\
v_0(x)
       \end{pmatrix}
\end{equation}
The term $F$ belongs to $ L^2(0, T; [L^2(0,1)]^2)$. Thus, by using a classical result, we obtain a well-posedness result for \eqref{linear4}-\eqref{linear5} giving solutions in $C^1([0, T];L^2(0,1))$.
Going back to $(u,v)$ we obtain the desired result (see \cite[pag 9]{komornik} for a similar argument).
\end{proof}

\subsubsection{\textbf{Lack of Controllability}}

In this part, we study the exact controllability properties of the linearized system

\begin{equation}\label{linear}
\begin{cases}
\partial_{t}u-\partial_t\partial_x^2u-\partial_xv=0, & (x,t) \in (0,1)\times (0,T), \\ 
\partial_{t}v-\partial_xu=0,  & (x,t) \in (0,1)\times (0,T), \\
u(0,t)=0,  \quad u(1,t)=h(t), & t \in (0,T), \\
\partial_x v(0,t)=0,  \quad \partial_xv(1,t)=0, & t \in (0,T), \\
u(x,0)=u_0(x), 	\quad v(x,0)=v_0(x), & x \in (0,1),
\end{cases}
\end{equation}
where $h$ is the control belongs to $H^1(0,T)$, and  the initial data $(u_0,v_0) \in H^1(0,1)\times L^2(0,1)$. 

Let $(\varphi, \psi)$ the solution of the adjoint system associated to \eqref{linear} given by 
\begin{equation}\label{linear6}
\begin{cases}
\partial_{t}\varphi-\partial_t\partial_x^2\varphi-\partial_x\psi=0, & (x,t) \in (0,1)\times (0,T), \\ 
\partial_{t}\psi-\partial_x\varphi=0,  & (x,t) \in (0,1)\times (0,T), \\
\varphi(0,t)= \varphi(1,t)=0, & t \in (0,T), \\
\partial_x \psi(0,t)=\partial_x \psi (1,t)=0, & t \in (0,T), \\
\varphi(x,T)=\varphi_T(x), 	\quad \psi(x,T)=\psi_T(x), & x \in (0,1),
\end{cases}
\end{equation}
where $(\varphi_T,\psi_T)$ belongs to $H^2(0,1)\cap H_0^1(0,1)\times H^1_0(0,1)$.

The following lemma gives a relation between the null controllability of 	\eqref{linear} with a control problem into a moment problem.

\begin{lemma}\label{lem1}
The initial data $(u_0,v_0) \in H^1_0(0,1)\times L^2(0,1)$  is controllable to zero in time $T>0$, with controls $h(\cdot)$  in $H^1(0,T)$, if and only if 
\begin{equation}\label{conditioncontrol}
 \left( u_0(\cdot), \varphi(\cdot,0) \right)_{H^1_0(0,1)}+\left( v_0(\cdot), \varphi(\cdot,0) \right)_{L^2(0,1)} 
 =  -\int_0^T h(t) \left(  \partial_x   \partial_t \varphi (1,t) +\psi (1,t) \right) dt,
\end{equation}
for any solution $(\varphi,\psi)$ of \eqref{linear6} and initial data $(\varphi_T,\psi_T)$ belongs to $H_0^1(0,1)\times L^2(0,1).$
\end{lemma}
\begin{proof}
Multiplying the first and second equation of the system \eqref{linear} by $\varphi$ and $\psi$ solution of \eqref{linear6}, respectively and integrating  over $(0,1)\times (0,T)$, it follows that 
\begin{align*}
\int_0^T\int_0^1  \left(\partial_{t}u-\partial_t\partial_x^2u-\partial_xv\right) \varphi dxdt + \int_0^T\int_0^1\left(\partial_{t}v-\partial_xu\right)\psi dx dt  =0
\end{align*}
Thus, integrating by parts, we have
\begin{multline*}
0=-\int_0^T\int_0^1  u \left( \partial_t \varphi -\partial_x^2  \partial_t \varphi -\partial_x\psi \right)  dxdt - \int_0^T\int_0^1  v \left( \partial_{t}\psi- \partial_x \varphi \right) dxdt \\
+ \int_0^1 \left[u\varphi     - \partial_x^2u  \varphi   +   v \psi \right]_0^Tdx   -  \int_0^T \left[ v  \varphi    +  u \partial_x   \partial_t \varphi      -  \partial_xu   \partial_t \varphi + u \psi \right]_0^1 dt.
\end{multline*}
Thus,  we deduce that
\begin{equation*}
 \int_0^1 \left[u\varphi     - \partial_x^2u  \varphi   +   v \psi \right]_0^Tdx   =  \int_0^T \left[ v  \varphi    +  u \partial_x   \partial_t \varphi      -  \partial_xu   \partial_t \varphi + u \psi \right]_0^1 dt.
\end{equation*}
Then,
\begin{multline*}
 \int_0^1  \left( u(x,T) \varphi_T(x)     +\partial_xu(x,T) \partial_x  \varphi_T(x)   +   v(x,T) \psi_T(x) \right) dx - \left[ \partial_xu(x,T)  \varphi_T(x)\right]_0^1\\
 -  \int_0^1  \left( u_0(x) \varphi(x,0)     + \partial_xu_0(x)  \partial_x\varphi(x,0)   +   v_0(x) \psi(x,0) \right) dx  +\left[ \partial_xu_0(x)  \varphi(x,0)\right]_0^1      \\
 =  \int_0^T \left( h(t) \partial_x   \partial_t \varphi (1,t)     + h(t)\psi (1,t) \right) dt,
\end{multline*}
thus, it follows that 
\begin{multline*}
\left( u(\cdot,T), \varphi_T(\cdot) \right)_{H^1_0(0,1)}+\left( v(\cdot,T), \varphi_T(\cdot) \right)_{L^2(0,1)} - \left( u_0(\cdot), \varphi(\cdot,0) \right)_{H^1_0(0,1)}-\left( v_0(\cdot), \varphi(\cdot,0) \right)_{L^2(0,1)} \\  
 =  \int_0^T h(t) \left(  \partial_x   \partial_t \varphi (1,t)+\psi (1,t) \right) dt,
\end{multline*}
for all $(\varphi_T, \psi_T) \in H^1_0(0,1)\times L^2(0,1)$. Then, $(u_0,v_0) \in H^1_0(0,1)\times L^2(0,1)$  is controllable to zero in time $T>0$ if and only if \eqref{conditioncontrol} is satisfied. 
\end{proof}

\begin{lemma}\label{lem2}
Let us consider initial data $U_0:=(u_0,v_0)$  belongs to $H^1_0(0,1)\times L^2(0,1)$ such that 
\begin{align*}
U_0=\sum_{n\in \Z^*}a_nU_n.
\end{align*}
Then, $U_0$  is controllable to zero in time $T>0$, with controls $h(\cdot)$ in $H^1(0,T)$, if and only if 
\begin{equation}
   \int_0^T h(t) e^{-i \lambda_n t}dt=   \frac{(-1)^n \sqrt{2} n^2\pi^2 a_n}{2\sqrt{1+n^2\pi^2}(n\pi \lambda_n + \sgn(n)\sqrt{1+n^2\pi^2})}
\end{equation}
\end{lemma}
\begin{proof}
Let us put $V_T:=(\varphi_T, \psi_T) \in H^1_0(0,1)\times L^2(0,1)$ such that $\quad V_T=\sum_{n\in \Z^*}b_nU_n$ and we use \eqref{conditioncontrol}. By using the spectral properties of the differential operator $A$  and the Proposition \ref{spectralsolution}, we have that 
\begin{align*}
U(x,t)&:=(u(x,t),v(x,t))=\sum_{n\in \Z^*} a_nU_ne^{-i \lambda_n t},\\
V(x,t)&:=(\varphi(x,t),\psi(x,t))=\sum_{n\in \Z^*} b_nU_ne^{i \lambda_n (T-t)}
\end{align*}
such that 
\begin{equation}
u(x,t)=\sum_{n\in \Z^*} \frac{i\sqrt{2}}{2\sqrt{1+n^2\pi^2}}a_n\sin(n\pi x)e^{-i \lambda_n t },  \quad v(x,t)= \sum_{n\in \Z^*}  \frac{\sqrt{2}\sgn(n)}{2}a_n   \cos (n\pi x) e^{-i \lambda_n t},
\end{equation}
and
\begin{equation}
\varphi(x,t)=\sum_{n\in \Z^*} \frac{i\sqrt{2}}{2\sqrt{1+n^2\pi^2}}b_n\sin(n\pi x)e^{i \lambda_n(T- t)},  \quad \psi(x,t)= \sum_{n\in \Z^*}  \frac{\sqrt{2}\sgn(n)}{2}b_n   \cos (n\pi x) e^{i \lambda_n (T-t)},
\end{equation}
where
$\lambda_n= \frac{\sgn(n)n\pi }{\sqrt{ 1+ n^2\pi^2}},$ for $n \in \Z^*$. It follows that 
\begin{align*}
 \left( u_0(\cdot), \varphi(\cdot,0) \right)_{H^1_0(0,1)}+\left( v_0(\cdot), \varphi(\cdot,0) \right)_{L^2(0,1)} 
& = \sum_{n\in \Z^*} \frac{-1}{2(1+n^2\pi^2)}a_nb_n e^{i \lambda_n T}+\sum_{n\in \Z^*} \frac{1}{2}a_nb_n e^{i \lambda_n T} \\
& = \sum_{n\in \Z^*} \frac{n^2\pi^2 a_nb_n e^{i \lambda_n T}}{2(1+n^2\pi^2)}
\end{align*}
and
\begin{multline*}
  \int_0^T h(t) \left(  \partial_x   \partial_t \varphi (1,t) +\psi (1,t) \right) dt \\
  =   \int_0^T h(t) \left[ \sum_{n\in \Z^*} \frac{\sqrt{2} n\pi \lambda_n}{2\sqrt{1+n^2\pi^2}}b_n\cos(n\pi )e^{i \lambda_n(T- t)}+\sum_{n\in \Z^*}  \frac{\sqrt{2}\sgn(n)}{2}b_n   \cos (n\pi ) e^{i \lambda_n ](T-t)}   \right]dt \\
  =  \frac{\sqrt{2}}{2} \sum_{n\in \Z^*} b_ne^{i \lambda_n T} \int_0^T h(t) \left[  \frac{ n\pi\lambda_n}{\sqrt{1+n^2\pi^2}}+ \sgn(n)   \right]\cos(n\pi )e^{-i \lambda_n t}dt.
\end{multline*}
From \eqref{conditioncontrol} we deduce that
\begin{equation}
\sum_{n\in \Z^*}b_ne^{i \lambda_n T} \left( \frac{n^2\pi^2 a_n}{1+n^2\pi^2}+\sqrt{2}   \int_0^T h(t) \left[  \frac{ n\pi\lambda_n}{\sqrt{1+n^2\pi^2}}+ \sgn(n)   \right]\cos(n\pi )e^{-i \lambda_n t}dt\right)=0,
\end{equation}
for any $\{b_n\}_{n\neq 0}$ in $l^2$. 
\end{proof}

In the next result, we will establish that the system \eqref{linear} is not spectrally controllable.
This means that no finite linear nontrivial combination of eigenvectors of the differential operator $A$ can be driven
to zero in finite time by using a control $h \in  H^1(0, T)$. Thus, from above lemmas the following negative result 
can be easily proved.

\subsection*{\bf Proof of Theorem \ref{noncontrollable}}
With the above results in hands, it is sufficently prove the following claim
\begin{claim}
No eigenfunction of the differential operator $A$ can be driven to zero in finite time.
\end{claim}

In fact, from Lemma \ref{lem2}, the null controllability of $U_m=(u_m,v_m)$  is equivalent to find $h \in H^1(0,T)$ such that 
\begin{equation}\label{eqq2}
   \int_0^T h(t) e^{-i \lambda_n t}dt=  
   \begin{cases}
    \dfrac{(-1)^n \sqrt{2} n^2\pi^2 a_n}{2\sqrt{1+n^2\pi^2}(n\pi \lambda_n + \sgn(n)\sqrt{1+n^2\pi^2})}  & n=m, \\ 
    \\
    0 & \forall n \in \Z^*, \,\, n \neq m.
    \end{cases}
\end{equation}
Suppose that there exists a function $h(\cdot)$ satisfying \eqref{eqq2}. Consider the function $f(\cdot) \in H^1 \left(-\frac{T}{2}, \frac{T}{2}\right)$ given by 
\begin{equation}
\text{$f(t)=h \left( \frac{T}{2}-t \right) $ almost everywhere in $\left(-\frac{T}{2}, \frac{T}{2}\right).$}
\end{equation}
Then, it yields 
\begin{multline*}
 \int_{-\frac{T}{2}}^{\frac{T}{2}} f(t) e^{i \lambda_n t}dt= e^{\frac{iT}{2}(\lambda_n- \lambda_m) } \int_{0}^{T} h(t) e^{-i \lambda_n t}dt \\
 =  
   \begin{cases}
    \dfrac{(-1)^n \sqrt{2} n^2\pi^2   }{2\sqrt{1+n^2\pi^2}(n\pi \lambda_n + \sgn(n)\sqrt{1+n^2\pi^2})}  & n=m, \\ 
    \\
    0 & \forall n \in \Z^*, \,\, n \neq m.
    \end{cases}
\end{multline*}
Thus, considering $f_m(t)=c_mf(t)$, where 
\begin{align*}
c_m(t)=\sqrt{2}(-1)^m   \left( \lambda_m   \sqrt{1+ \frac{1}{m^2\pi^2}} + \sgn(m) \left(   1+ \frac{1}{m^2\pi^2} \right)\right) , 
\end{align*}
it follows that 
\begin{equation*}
 \int_{-\frac{T}{2}}^{\frac{T}{2}} f_m(t) e^{i \lambda_n t}dt=
   \begin{cases}
    1 & n=m, \\ 
    \\
    0 & \forall n \in \Z^*, \,\, n \neq m.
    \end{cases}
\end{equation*}
It is a contradiction. Indeed, one can see that is not possible due the  following  Theorem \ref{notpossible}  which shows  a negative property related with the existence of a biorthogonal family. 

$\hfill\square$

\begin{theorem}\label{notpossible}
Let $T>0$ and $m\in \Z^*$. There is not function $f_m \in L^2(-T,T)$ such that 
\begin{equation}\label{nopossible1}
\int_{-T}^T f_m(t)e^{i\lambda_n t}dt = 
\begin{cases}
1 & \text{if $n= m$, } \\
0 & \text{if $n \in \Z^*, n\neq m$,}
\end{cases}
\end{equation}
where $\{i\lambda_n\}$, $n\in \Z^*$ are the eigenvalues of the differential operator $A$. 
\end{theorem}
\begin{proof}
Suppose that there exists a function $f_m$ belongs to $L^2(-T,T)$ for some $m \in \Z^*$,  such that
\eqref{nopossible1} is satisfied. Thus, let us define $F_m :\C \rightarrow \C$ by
\begin{align*}
F_m(z):=\int_{-T}^T f_m(t)e^{iz t}dt
\end{align*}
By using the Paley–Wiener theorem, it follows that  $F_m$ is an entire function and from \eqref{nopossible1}, it yields
\begin{align}\label{nopossible2}
F_m(\lambda_m)=1.
\end{align}
Since ·$\lambda_n=\frac{|n|\pi }{\sqrt{1+n^2\pi^2 }}$, we have that  $\lim_{|n|\rightarrow \infty }\lambda_n=1$. Then, $F$ is zero on a set with a finite accumulation  point. Therefore $F = 0$ which contradicts \eqref{nopossible2}. Hence, there is no function $f_m \in L^2(-T,T)$ such that \eqref{nopossible1}.
\end{proof}

\subsection{Approximated Controllability}
In spite of the lack of exact controllability from the boundary of system, we will prove that system \eqref{linear} is approximately controllable, i,e the set of reachable states 
\begin{align}\label{finalstate}
R(T,(u_0,v_0))=\left\lbrace (u(T,x),v(T,x)): h \in H^1(0,T)\right\rbrace
\end{align}
is dense in $H^{1}_0(0,1) \times L^2(0,1)$ for any $(u_0,v_0) \in H^{1}_0(0,1) \times L^2(0,1)$ and $T>0.$  This property is equivalent to a unique continuation property for the adjoint system associated.  The main idea is to use the series expansion of the solution in terms of the eigenvectors of the operator in order to reduce the problem to a unique continuation problem (of the eigenvectors). 
 
 \subsection*{\bf Proof of Theorem \ref{approximated}}
 Due to the linearity of the system under consideration, it is sufficient to prove the result for any $T > 0$ and $(u_0,v_0)=(0,0)$. Thus, we are going to prove that the set \eqref{finalstate} is dense in $H^1_0(0,1)\times L^2(0,1).$
 
 Let $(u,v) \in C([0,T];H^1(0,1)\times L^2(0,1))$ the corresponding solution of the system \eqref{linear} given by Theorem 
 \ref{existence1} and $(\varphi,\psi)$ solution of the adjoint system \eqref{linear6}. Then, it follows
that
\begin{equation}\label{final1}
 \left( u(T,\cdot), \varphi_T(\cdot) \right)_{H^1_0(0,1)}+\left( v(T,\cdot), \varphi_T(\cdot) \right)_{L^2(0,1)} 
 =  \int_0^T h(t) \left(  \partial_x   \partial_t \varphi (1,t) +\psi (1,t) \right) dt.
\end{equation}
Suppose that $R(T,(u_0,v_0))$ is not dense in $H^1_0(0,1)\times L^2(0,1).$ In this case, there exists $(\varphi_T,\psi_T) \neq (0,0)$ belongs to $ H^1_0(0,1)\times L^2(0,1)$ such that 
\begin{align*}
\left( u(T,\cdot), \varphi_T(\cdot) \right)_{H^1_0(0,1)}+\left( v(T,\cdot), \varphi_T(\cdot) \right)_{L^2(0,1)} =0, \quad \forall h \in H^1(0,T).
\end{align*}
Consequently, from \eqref{final1}, we obtain
\begin{equation}\label{final2}
 \int_0^T h(t) \left(  \partial_x   \partial_t \varphi (1,t) +\psi (1,t) \right) dt=0,  \quad h \in H^1(0,T).
\end{equation}
Thus, 
\begin{align*}
\partial_x   \partial_t \varphi (1,t) +\psi (1,t)=0,  \quad \forall t \in [0,T].
\end{align*}
On the other hand, since $A$ is a skew-adjoint operator in $H^1_0(0,1)\times L^2(0,1)$, it has a
sequence of eigenvalues $\{\lambda_n\}_{n\in \Z^*} \subset  i\R$ and the corresponding eigenfunctions $\{U_n\}_{n\in \Z^*}$ form a orthonormal basis of ($H^1_0(0,1)\times L^2(0,1)$. Then, if $V_T:=(\varphi_T, \psi_T)$ belongs to $H^1_0(0,1)\times L^2(0,1)$, we have
\begin{align*}
V_T=\sum_{n\in \Z^*}b_nU_n 
\end{align*}
and the corresponding solution $(\varphi, \psi)$ of the adjoint system can be written as
\begin{align*}
V(x,t)&:=(\varphi(x,t),\psi(x,t))=\sum_{n\in \Z^*} b_nU_n(x)e^{i \lambda_n (T-t)}
\end{align*}
thus, it follows that 
\begin{align*}
\varphi(x,t) &=\sum_{n\in \Z^*} b_nU_n^1(x)e^{i \lambda_n (T-t)} = \frac{i\sqrt{2}}{2}\sum_{n\in \Z^*} b_n  \frac{\sin(n\pi x)  e^{i \lambda_n (T-t)}}{\sqrt{1+n^2\pi^2}}  \\
\psi(x,t)&=\sum_{n\in \Z^*} b_nU_n^2(x)e^{i \lambda_n (T-t)}= \frac{\sqrt{2}}{2}\sum_{n\in \Z^*} b_n \sgn(n)\cos (n\pi x)   e^{i \lambda_n (T-t)}
\end{align*}
Then, we obtain that 
\begin{align*}
\partial_x   \partial_t \varphi(1,t) &  = \frac{\sqrt{2}}{2}\sum_{n\in \Z^*} b_n  \frac{n \pi \lambda_n (-1)^{|n|}  e^{i \lambda_n (T-t)}}{\sqrt{1+n^2\pi^2}}  
\end{align*}
and 
\begin{align*}
\psi(1,t)&=\frac{\sqrt{2}}{2}\sum_{n\in \Z^*} b_n \sgn(n)(-1)^{|n|}    e^{i \lambda_n (T-t)}
\end{align*}
Thus, it yields, 
\begin{align*}
0=\partial_x   \partial_t \varphi(1,t)+ \psi(1,t) &=\frac{\sqrt{2}}{2}\sum_{n\in \Z^*} (-1)^{|n|}  b_n \left(  \frac{n \pi \lambda_n  }{\sqrt{1+n^2\pi^2}}  + \sgn(n) \right)    e^{i \lambda_n (T-t)}.
\end{align*}
Since $\varphi$ and $\psi$ are an analytic function (see Theorem 2.1), we can integrate the identity above over $(-S, S)$, for any $S > 0$. Then, for each $m \in \Z^*$, we deduce that
\begin{align*}
0=\lim_{S\rightarrow \infty} \frac{1}{S}\int_{-S}^S \partial_x   \partial_t \varphi(1,t)+ \psi(1,t) e^{-\lambda_m t}dt
\end{align*}
Thus, 
\begin{align*}
\lim_{S\rightarrow \infty}\frac{1}{S}\int_{-S}^S \left[\sum_{n\in \Z^*} (-1)^{|n|}  b_n \left(  \frac{n \pi \lambda_n  }{\sqrt{1+n^2\pi^2}}  + \sgn(n) \right)    e^{i \lambda_n (T-t)} \right] e^{-\lambda_m t}dt=0.
\end{align*}
Hence, it follows that $b_m=0$ for all $m\in \Z^*$ and 
therefore $V_T:=(\varphi_T, \psi_T)=(0,0)$ , which represents a contradiction. Then, $R(T,(u_0,v_0))$
is dense in $H^{1}_0(0,1) \times L^2(0,1)$. 
 $\hfill\square$
 
\subsection*{\bf Acknowledgments} 
The authors thank the anonymous referees for their very valuable comments and suggestions.
The first author was supported by Mathamsud-21-Math-03  and the 100.000 Strong in the Americas Innovation Fund. The second author was partially supported by Mathamsud 21-MATH-03.  This work was carried out during visits of the authors to the Universidad Nacional de Colombia sede Manizales. The authors would like to thank the universities for its hospitality.


\begin{thebibliography}{8}




\bibitem{albert}
 J. Albert, On the decay of solutions of the generalized Benjamin-Bona-Mahony equations, Journal of
Mathematical Analysis and Applications, vol. 141, no. 2, pp. 527–537, 1989.


\bibitem{CB} {C. Banquet, The Symmetric Regularized-Long-Wave Equation: Well-posedness and Nonlinear
Stability, \em Physica D,} {\bf 241} (2012), 125 - 133.

\bibitem{bbm}
T. B. Benjamin, J. L. Bona and J. J. Mahony, Model equations for long waves in nonlinear dispersive systems, Phil. Trans. Royal Soc. London, A 227,  47-78, (1972).

\bibitem{B-T} {J. L. Bona and N. Tzvetkov, Sharp Well-posedness Results for the BBM equation, \em Discrete and Continuous Dynamical System,} {\bf 23}
(2009), 1241 - 1252.


\bibitem{castro}
C. Castro,. Exact controllability of the 1-d wave equation from a moving interior point. ESAIM: Control, Optimisation and Calculus of variations, 19(1), 301-316. (2013)

\bibitem{BPS1983}
J. L. Bona, W. G. Pritchard and L. R. Scott, A comparison of solutions of two model equations for long waves, Fluid Dynamics in Astrophysics and Geophysics, Norman R. Lebovitz, ed., Lectures in Applied Mathematics, 20, 235-267, (1983).

\bibitem{BPS1981}
J. L. Bona, W. G. Pritchard and L. R. Scott, An evaluation of a model equation for water waves, Phil. Trans. Royal Soc. London, A 302,  457-510, (1981).


\bibitem{brezis}
H. Brezis, Analyse Fonctionnelles: théorie et applications, Masson, Paris, 1987.

\bibitem{cerparivas2018}
E. Cerpa, I. Rivas, On the controllability of the Boussinesq equation in low regularity, Journal of Evolution Equations, to appear

\bibitem{cerpa2018}
E. Cerpa and E. Crépeau. On the controllability of the improved Boussinesq equation. SIAM J. Control Optim.,  vol. 56, no 4, p. 3035-3049, 2018. 


\bibitem{cerpa2003}
E. Crépeau. Exact controllability of the Boussinesq equation on a bounded domain. Differential Integral Equations, 16(3), 303-326, 2003


\bibitem{Chen1} 
{L. Chen, Stability and instability of solitary waves for generalized
symmetric-regularized-long-wave equation, \em Physica D,} {\bf 118}
(1998), 53 - 68.

\bibitem{C-L} 
{Y. Chen and B. Li, Travelling wave solutions for generalized symmetric regularized-long-wave
equations with high-order nonlinear terms, \em Chinese Phys.}, {\bf
13} (2004), 302 - 306.

\bibitem{F} 
{X. Fei, Application of Exp-function method to Symmetric Regularized Long Wave (SRLW) equation,
\em Physics Lett. A,} {\bf 372} (2008), 252 - 257.

\bibitem{komornik}
V. Komornik and P. Loreti, Fourier series in control theory, Springer Monographs, in Mathematics, Springer-Verlag, New York, 2005.

\bibitem{leu1989}
 G. Leugering, E. J. P. G. Schmidt, Boundary control of a vibrating plate with internal damping, Mathematical
Methods in the Applied Sciences, 11 (1989), 573–586.


\bibitem{lions}
J.-L. Lions. Contrôlabilité exacte, perturbations et stabilisation de systèmes distribués. Tome 2,
volume 9 of Recherches en Mathématiques Appliquées [Research in Applied Mathematics]. Masson,
Paris, 1988. Perturbations. [Perturbations]

\bibitem{lions_3}
 J.-L. Lions, Pointwise control for distributed systems, in Control and estimation in distributed parameter
systems, edited by H. T. Banks, SIAM, 1992.


\bibitem{lions_1}
J.-L. Lions and E. Magenes, Problèmes aux Limites non Homogènes et Applications, Vol. 1,
Dunod, Paris, 1968.

\bibitem{lions_2}
 J.-L. Lions and E. Magenes, Problèmes aux Limites non Homogènes et Applications, Vol. 2,
Dunod, Paris, 1968.




\bibitem{micu}
S. Micu, On the controllability of the linearized Benjamin–Bona–Mahony equation, SIAM J. Control Optim. 39 (6) 1677–1696, (2001).

\bibitem{micu2009}
S. Micu, J. Ortega, L. Rosier, B.-Y. Zhang, Control and stabilization of a family of Boussinesq systems, Discrete Contin. Dyn.
Syst. 24 (2) (2009) 273–313

\bibitem{P1996}
D. H. Peregrine, Calculations of the development of an undular bore, J. Fluid Mechanics 25, 321-330,  (1996).

\bibitem{p1967}
D. H. Peregrine, Long waves on a beach, J. Fluid Mechanics 27, 815-827, (1967).

\bibitem{rosier}
L. Rosier and B.Y. Zhang, Unique continuation property and control for the Benjamin-Bona-Mahony equation on the torus. Journal of Differential Equations, 254, 141-178, (2013).


\bibitem{R} 
{D. Roum\'egoux, A simpletic non-queezing theorem for BBM
equation, \em Dynamics of PDE,} {\bf 7} (2010), 289 - 305.

\bibitem{rusell1967}
D. L. Russell, Nonharmonic Fourier series in the control theory of distributed parameter sytems,  {\em J. Math. Anal. Appl.,} 18 (1967), pp. 542--560.

\bibitem{rusell1985}
 D. L. Russell, Mathematical models for the elastic beam and their control-theoretic implications, in H. Brezis,
M. G. Crandall and F. Kapper (eds), Semigroup Theory and Applications, Longman, New York (1985).


\bibitem{scott1973}
A.C. Scott, F.Y.F. Chu and D.W. McLaughlin, The soliton: A new concept in applied science, Proc. IEEE 61, 1443-1483, (1973).


\bibitem{S-F} 
{C. Seyler and D. Fenstermacher. A symmetric regularized-long-wave equation, \em Phys.
Fluids,} {\bf 27} (1984), 4 - 7.

\bibitem{shang}
Y. Shang, Unique continuation for the symmetric regularized long wave equation. Mathematical methods in the applied sciences 30.4 (2007): 375-388.


\bibitem{zhang}
B.-Y. Zhang. Exact controllability of the generalized Boussinesq equation. In Control and estimation
of distributed parameter systems (Vorau, 1996), volume 126 of Internat. Ser. Numer. Math., pages
297-310. Birkhäuser, Basel, 1998.




\end{thebibliography}
\end{document}